\documentclass{svjour3}                     
\smartqed  
\usepackage{graphicx}
\usepackage{amsmath,amssymb}
\usepackage{hyperref}
\usepackage{enumerate}
\usepackage{mathtools}
\usepackage{xcolor}

\usepackage{makeidx}
\makeindex

\newcommand{\fcap}{{f_{\scalebox{0.5}{$\mathrm{C}$}}}}
\newcommand{\find}{{f_{\scalebox{0.5}{$\mathrm{L}$}}}}
\newcommand{\fres}{{f_{\scalebox{0.5}{$\mathrm{R}$}}}}
\newcommand{\dfcap}{{m_{\scalebox{0.5}{$\mathrm{C}$}}}}
\newcommand{\dfind}{{m_{\scalebox{0.5}{$\mathrm{L}$}}}}
\newcommand{\dfres}{{m_{\scalebox{0.5}{$\mathrm{R}$}}}}
\newcommand{\xfcap}{{\chi_{\scalebox{0.5}{$\mathrm{C}$}}}}
\newcommand{\xfind}{{\chi_{\scalebox{0.5}{$\mathrm{L}$}}}}
\newcommand{\xfres}{{\chi_{\scalebox{0.5}{$\mathrm{R}$}}}}
\newcommand{\gcap}{{g_{\scalebox{0.5}{$\mathrm{C}$}}}}
\newcommand{\gind}{{g_{\scalebox{0.5}{$\mathrm{L}$}}}}
\newcommand{\gres}{{g_{\scalebox{0.5}{$\mathrm{R}$}}}}
\newcommand{\strongfcap}{{F_{\scalebox{0.5}{$\mathrm{C}$}}}}
\newcommand{\strongfind}{{F_{\scalebox{0.5}{$\mathrm{L}$}}}}

\newcommand{\isrc}{{i_{\scalebox{0.5}{$\mathrm{S}$}}}}
\newcommand{\vsrc}{{v_{\scalebox{0.5}{$\mathrm{S}$}}}}

\newcommand{\xcap}{{x_{\scalebox{0.5}{$\mathrm{C}$}}}}
\newcommand{\xind}{{x_{\scalebox{0.5}{$\mathrm{L}$}}}}
\newcommand{\xres}{{x_{\scalebox{0.5}{$\mathrm{R}$}}}}
\newcommand{\icap}{{i_{\scalebox{0.5}{$\mathrm{C}$}}}}
\newcommand{\iind}{{i_{\scalebox{0.5}{$\mathrm{L}$}}}}
\newcommand{\ires}{{i_{\scalebox{0.5}{$\mathrm{R}$}}}}
\newcommand{\ivsrc}{{i_{\scalebox{0.5}{$\mathrm{V}$}}}}

\newcommand{\vcap}{{v_{\scalebox{0.5}{$\mathrm{C}$}}}}
\newcommand{\vind}{{v_{\scalebox{0.5}{$\mathrm{L}$}}}}
\newcommand{\vres}{{v_{\scalebox{0.5}{$\mathrm{R}$}}}}

\newcommand{\ipcap}{{i'_{\scalebox{0.5}{$\mathrm{C}$}}}}

\newcommand{\vpind}{{v'_{\scalebox{0.5}{$\mathrm{L}$}}}}
\newcommand{\vpres}{{v'_{\scalebox{0.5}{$\mathrm{R}$}}}}

\newcommand{\Acap}{{A_{\scalebox{0.5}{$\mathrm{C}$}}}}
\newcommand{\Aind}{{A_{\scalebox{0.5}{$\mathrm{L}$}}}}
\newcommand{\Ares}{{A_{\scalebox{0.5}{$\mathrm{R}$}}}}
\newcommand{\Acsrc}{{A_{\scalebox{0.5}{$\mathrm{I}$}}}}
\newcommand{\Avsrc}{{A_{\scalebox{0.5}{$\mathrm{V}$}}}}
\newcommand{\Axtype}{{A_{\scalebox{0.5}{$\mathrm{X}$}}}}
\newcommand{\Aytype}{{A_{\scalebox{0.5}{$\mathrm{Y}$}}}}
\newcommand{\Atcap}{{A^\top_{\scalebox{0.5}{$\mathrm{C}$}}\!}}
\newcommand{\Atind}{{A^\top_{\scalebox{0.5}{$\mathrm{L}$}}\!}}
\newcommand{\Atres}{{A^\top_{\scalebox{0.5}{$\mathrm{R}$}}\!}}

\newcommand{\Atvsrc}{{A^\top_{\scalebox{0.5}{$\mathrm{V}$}}\!}}
\newcommand{\Pcapvsrc}{{P_{\scalebox{0.5}{$\mathrm{CV}$}}}}
\newcommand{\Presmcapvsrc}{{P_{\scalebox{0.5}{$\mathrm{R-CV}$}}}}
\newcommand{\Qcapvsrc}{{Q_{\scalebox{0.5}{$\mathrm{CV}$}}}}
\newcommand{\Qresmcapvsrc}{{Q_{\scalebox{0.5}{$\mathrm{R-CV}$}}}}

\newcommand{\Ptresmcapvsrc}{{P^\top_{\scalebox{0.5}{$\mathrm{R-CV}$}}}}
\newcommand{\Qtcapvsrc}{{Q^\top_{\scalebox{0.5}{$\mathrm{CV}$}}}}
\newcommand{\Qtresmcapvsrc}{{Q^\top_{\scalebox{0.5}{$\mathrm{R-CV}$}}}}
\newcommand{\PLIcut}{{P_{\scalebox{0.5}{$\mathrm{LI-cut}$}}}}
\newcommand{\QLIcut}{{Q_{\scalebox{0.5}{$\mathrm{LI-cut}$}}}}
\newcommand{\PtLIcut}{{P^\top_{\scalebox{0.5}{$\mathrm{LI-cut}$}}}}

\newcommand{\PCVloop}{{P_{\scalebox{0.5}{$\mathrm{CV-loop}$}}}}
\newcommand{\QCVloop}{{Q_{\scalebox{0.5}{$\mathrm{CV-loop}$}}}}

\newcommand{\QtCVloop}{{Q^\top_{\scalebox{0.5}{$\mathrm{CV-loop}$}}}}

\newcommand{\dxdy}[2]{\frac{\mathrm{d}#1}{\mathrm{d}#2}}
\newcommand{\im}{\mathrm{im}}
\newcommand{\rank}{\mathrm{rank}}
\newcommand{\setR}{\mathbb{R}}

\newcommand{\Efield}{\vec{E}}
\newcommand{\Dfield}{\vec{D}}
\newcommand{\Hfield}{\vec{H}}
\newcommand{\Bfield}{\vec{B}}
\newcommand{\Jfield}{\vec{J}}
\newcommand{\Afield}{\vec{A}}
\newcommand{\rhofield}{\rho}
\renewcommand{\varphi}{\phi}

\newcommand{\divfit}{{S}}
\newcommand{\divdfit}{\widetilde{{S}}}
\newcommand{\curlfit}{{C}}
\newcommand{\curldfit}{\widetilde{{C}}}
\newcommand{\gradfit}{{G}}

\newcommand{\Meps}{{M}_{\varepsilon}}
\newcommand{\Msigma}{{M}_{\sigma}}
\newcommand{\Mmu}{{M}_{\mu}}
\newcommand{\Mnu}{{M}_{\nu}}

\spnewtheorem{assumption}{Assumption}{\bf}{\it}

\journalname{Progress in Differential-Algebraic Equations}

\begin{document}

\title{Generalized Circuit Elements
  \thanks{
    Supported by Deutsche Forschungsgemeinschaft
    (DFG, German Research Foundation) under Germany's Excellence Strategy:
    the Berlin Mathematics Research Center MATH+ (EXC-2046/1, project ID: 390685689) and
    the Graduate School of Computational Engineering at TU Darmstadt (GSC 233) as well as
    the DFG grant SCHO1562/1-2 and the Bundesministerium f\"ur Wirtschaft und Energie (BMWi, Federal Ministry for Economic Affairs and Energy) grant 0324019E.
  }
}

\titlerunning{Generalized Circuit Elements}

\author{Idoia Cortes Garcia \and
  Sebastian Sch\"ops \and
  Christian Strohm \and
  Caren Tischendorf
}

\authorrunning{Cortes Garcia, Sch\"ops, Strohm, Tischendorf} 

\institute{Idoia Cortes Garcia \at
  TU Darmstadt, Computational Electromagnetics Group,
	Schlossgartenstr. 8, 64289  Darmstadt\\
  \email{idoia.cortes@tu-darmstadt.de}
  \and
  Sebastian Sch\"ops \at
  TU Darmstadt, Computational Electromagnetics Group,
	Schlossgartenstr. 8, 64289  Darmstadt\\
  \email{sebastian.schoeps@tu-darmstadt.de}
  \and
  Christian Strohm \at
  Humboldt Universit\"at zu Berlin, Unter den Linden 6,
  10099 Berlin\\
  \email{strohmch@math.hu-berlin.de}
  \and
  Caren Tischendorf \at
  Humboldt Universit\"at zu Berlin, Unter den Linden 6,
  10099 Berlin\\
  \email{tischendorf@math.hu-berlin.de}
}

\date{Received: date / Accepted: date}

\maketitle

\begin{abstract}
The structural analysis, i.e., the investigation of the differential-algebraic nature, of circuits containing simple elements, i.e., resistances, inductances and capacitances is well established. However, nowadays circuits contain all sorts of elements, e.g. behavioral models or partial differential equations stemming from refined device modelling. This paper proposes the definition of generalized circuit elements which may for example contain additional internal degrees of freedom, such that those elements still behave structurally like resistances, inductances and capacitances. Several complex examples demonstrate the relevance of those definitions.
\end{abstract}

\section{Introduction}
\label{intro}
Circuits or electric networks are a common modeling technique to describe the electrotechnical behavior of large systems. Their structural analysis, i.e., the investigation of the properties of the underlying differential-algebraic equations (DAEs), has a long tradition. For example Bill Gear studied in 1971 `the mixed differential and algebraic equations of the type that commonly occur in the transient analysis of large networks' in \cite{Gear_1971aa}. At that time several competing formulations were used in the circuit simulation community, for example the sparse tableau analysis (STA) was popular. This changed with the introduction of the modified nodal analysis (MNA) by Ho et. al in \cite{Ho_1975aa} and the subsequent development of the code SPICE \cite{Nagel_1975aa}. Nowadays all major circuit simulation tools are using some dialect of MNA, e.g. the traditional formulation or the flux/charge oriented one \cite{Feldmann_1999aa}. The mathematical structure has been very well understood in the case of simple elements, i.e., resistances, inductances and capacitances as well as sources \cite{Gunther_1995aa,Estevez-Schwarz_2000aa}.

However, the complexity of element models has increased quickly. For example, the semiconductor community develops various phenomenological and physical models, which are standardized e.g. in the BSIM (Berkeley Short-channel IGFET Model) family, \cite{Sheu_1987aa}. The development of \textit{mixed-mode device simulation} has become popular, which is mathematically speaking the coupling of DAEs with partial differential equations (PDEs), e.g. \cite{Rollins_1988aa,Mayaram_1992aa,Grasser_2000aa,Gunther_2000ab}. Even earlier, low frequency engineers have established \textit{field-circuit-coupling}, i.e., the interconnection of finite element machine models with circuits, first based on loop analysis, later (modified) nodal analysis e.g. \cite{Potter_1983aa,Costa_2000aa}. 

Until now, the structural DAE analysis of circuits which are based on complex (`refined') elements has mainly been carried out on a case by case basis, e.g. for elliptic semiconductor models in \cite{Ali_2003aa}, parabolic-elliptic models of electrical machines in \cite{Tsukerman_2002aa,Bartel_2011aa,Cortes-Garcia_2019aa} and hyperbolic models stemming from the full set of Maxwell's equations in \cite{Baumanns_2013aa}. Based on the analysis made in \cite{Cortes-Garcia_2019aa}, this contribution aims for a more systematic analysis: we consider each element as an arbitrary (smooth) function of voltages, currents, internal variables and their derivatives. Then, we formulate sets of assumptions (`generalized elements') on these functions, e.g. which quantity is derived or which DAE-index does the function have. Based on these assumptions we proof a DAE index result that generalizes \cite{Estevez-Schwarz_2000aa}. Not surprisingly, it turns out that our generalized elements are natural generalizations of the classical elements, i.e., resistances, inductances and capacitances. All results are formulated in the context of electrical engineering but the presented approach is also of interest for the analysis and simulation of other networks such as gas transport networks \cite{JanTis_2014aa,GruJHCTB_2014aa,BGHHST_2018aa} or power networks \cite{MMOS_2018aa}.

The paper is structured as follows: we start with a few basic mathematical definitions and results in Section~\ref{sec:analytic_preliminaries}, then we give the definitions of our generalized elements and some simple examples in Section~\ref{sec:generalized_circuit_elements}. In Section~\ref{sec:circuit_structures} we discuss the mathematical modeling of circuits by modified nodal analysis. Finally, \ref{sec:DAE_index} proves the new DAE index results which are then applied to several very complex refined models in Section~\ref{sec:refined}.

\section{Mathematical Prelimenaries}
\label{sec:analytic_preliminaries}

Let us collect some basic notations and definitions:
\begin{definition}
  A function $f:\mathbb{R}^m\to \mathbb{R}^m$ is called
  strongly monotone if and only if there is a constant $c>0$ such that
  \begin{align*}
    \forall\,x,\bar x\in\mathbb{R}^m:
    \quad \langle f(x)-f(\bar x),x-\bar x\rangle \ge c \|x-\bar x\|^2.
  \end{align*}    
\end{definition}

\begin{lemma}
  \label{lem:posdef.monotone}
  Let $M\in\mathbb{R}^{m\times m}$ be a matrix. Then,
  the linear function $f(x):=Mx$ is strongly monotone if and only if $M$ is positive definite.
\end{lemma}

\begin{proof}
  If $f(x):=Mx$ is strongly monotone we find a constant $c>0$ such that
  for all $x\in\mathbb{R}^m$ with $x\neq 0$
  \begin{align*}
    \langle Mx,x \rangle = \langle f(x)-f(0),x-0 \rangle \ge c \|x-0\| >0,
  \end{align*}
  that means $M$ is positive definite. Next, we show the opposite direction. Let $M$ be
  positive definite. We split $M$ into its symmetric and non-symmetric part
  \begin{align*}
    M=M_s+M_n, \quad M_s = \frac{1}{2}(M+M^\top),\quad M_n = \frac{1}{2}(M-M^\top).
  \end{align*}
  Consequently, for all $x\in\mathbb{R}^m$ with $x\neq 0$,
  \begin{align*}
    \langle M_sx,x \rangle = \langle Mx,x \rangle >0.
  \end{align*}
  Since $M_s$ is symmetric, we find a unitary matrix $T$ and a diagonal matrix $D$ such that
  $M_s = T^{-1}DT$. We get that
  \begin{align*}
    0< \langle M_sx,x \rangle = \langle T^{-1}DTx,x \rangle 
    = \langle DTx,Tx \rangle = \sum_{j=1}^m d_{jj}y_j^2 \quad\text{with}\quad y:=Tx.
  \end{align*}
  Choosing the unit vectors $y:=e_i$, we find that $d_{ii}>0$ for all $i=1,...,m$.
  Defining $c:=\min_{i=1,...,m} d_ii$, we see that for all $x\in\mathbb{R}^m$ 
  \begin{align*}
    \langle Mx,x \rangle = \langle M_sx,x \rangle \ge  \sum_{j=1}^m cy_j^2
    = c \|Tx\|^2 = c \|x\|^2.
  \end{align*}  
  Finally, we obtain, for any $x,\bar x\in\mathbb{R}^m$ 
  \begin{align*}
    \langle f(x)-f(\bar x),x-\bar x \rangle =
    \langle M(x-\bar x),x-\bar x \rangle \ge  c \|x-\bar x\|^2.
  \end{align*} \qed
\end{proof}

\medskip
\begin{definition}\label{def:stronmonomxn}
  A function $f:\mathbb{R}^m\times \mathbb{R}^n\to \mathbb{R}^m$ is called
  strongly monotone with respect to $x$ if and only if there is a constant $c>0$ such that
  \begin{align*}
    \forall\,y\in\mathbb{R}^n\ \forall\,x,\bar x\in\mathbb{R}^m:
    \quad \langle f(x,y)-f(\bar x,y),x-\bar x\rangle \ge c \|x-\bar x\|^2.
  \end{align*}    
\end{definition}

\medskip
\begin{remark}
	In case of variable matrix functions $M(y)$, the function $f(x,y):=M(y)x$ might be
	not strongly monotone with respect to $x$ even if $M(y)$ is positive definite for each $y$.
	For strong monotony, one has to ensure that the eigenvalues of the symmetric part of
	$M(y)$ can be bounded from below by a constant $c>0$ independent of $y$.
\end{remark}

\begin{lemma}
  \label{lem:solvability.monotone}
  Let $f=f(x,y):\mathbb{R}^m\times \mathbb{R}^n\to \mathbb{R}^m$ be strongly monotone
  with respect to $x$ and continuous. Then, there is a uniquely defined continuous
  function $g:\mathbb{R}^n\to\mathbb{R}^m$ such that $f(g(y),y)=0$ for all
  $y\in\mathbb{R}^n$.
\end{lemma}

\begin{proof}
  For fixed $y\in\mathbb{R}^n$ we define $F_y:\mathbb{R}^m\to\mathbb{R}^m$ by
  \begin{align*}
    F_y(x):=f(x,y)\quad\forall\,x\in\mathbb{R}^m.
  \end{align*}
  Since $f$ is strongly monotone with respect to $x$, the function $F_y$ is
  strongly monotone. The Theorem of Browder-Minty, e.g. \cite{zeidler2013nonlinear} and \cite{ortega1970iterative}, provides a 
  unique $z_y\in\mathbb{R}^m$ such that $F_y(z_y)=0$ and, hence, $f(z_y,y)=0$. We define
  $g:\mathbb{R}^n\to\mathbb{R}^m$ by
  \begin{align*}
    g(y):= z_y.
  \end{align*}
  Obviously, $f(g(y),y)=0$ for all $y\in\mathbb{R}^n$. It remains to show that $g$ is
  continuous. Let $(y_k)$ be a convergent series in $\mathbb{R}^n$ with
  $y_k\to y_*\in\mathbb{R}^n$ for $k\to\infty$. Since $f$ is strongly monotone with respect
  to $x$, there is a constant $c>0$ such that
  \begin{align*}
    \|g(y_k)-g(y_*)\|^2
    &\le \frac{1}{c} \langle f(g(y_k),y_k) - f(g(y_*),y_k), g(y_k)-g(y_*)\rangle\\
    &\le \frac{1}{c} \| f(g(y_k),y_k) - f(g(y_*),y_k)\|\, \|g(y_k)-g(y_*)\| \\
    &= \frac{1}{c} \| f(g(y_*),y_k)\|\, \|g(y_k)-g(y_*)\| \\
    &= \frac{1}{c} \| f(g(y_*),y_k) - f(g(y_*),y_*)\|\, \|g(y_k)-g(y_*)\|.
  \end{align*}
  Since $f$ is continuous, we may conclude that $g(y_k)\to g(y_*)$ for $k\to\infty$. \qed
\end{proof}

\begin{lemma}
  \label{lem:solutionform.monotone}
  Let $M\in\mathbb{R}^{m\times k}$ be a matrix and $P\in\mathbb{R}^{k\times k}$
  be a projector along $\ker\,M$. Additionally, let
  $f=f(x,y):\mathbb{R}^m\times \mathbb{R}^n\to \mathbb{R}^m$ be
  strongly monotone with respect to $x$ and continuous as well as
  $r:\mathbb{R}^n\to \mathbb{R}^m$ be a continuous function.
  Then, there is a continuous function
  $g:\mathbb{R}^n\to\mathbb{R}^k$ such that 
  \begin{align}
    \label{eq:equivalence}
    M^\top f(M z,y) + P^\top r(y) = 0 \quad\text{if and only if}\quad
    Pz = g(y).
  \end{align}
\end{lemma}

\begin{proof}
  In the degenerated case that $M=0$ we have $P=0$ and the zero function $g(y)\equiv 0$
  fulfills obviously the equivalence \eqref{eq:equivalence}.
  Let be $M\neq 0$ for the further considerations. We chose a 
  a basis $B$ of $\im\,P$. For $r:=\rank\, P$, we form the full-column rank matrix
  $\widetilde P\in\mathbb{R}^{k\times r}$  as a matrix whose columns
  consist of all basis vectors of $B$. By construction,
  $\ker M\widetilde P=\{0\}$ and, hence,
  the matrix $(M\widetilde P)^\top M\widetilde P$ is non-singular.
  Next, we introduce a function $F:\mathbb{R}^r\times \mathbb{R}^n\to\mathbb{R}^r$ by
  \begin{align*}
    F(u,y):= (M\widetilde P)^\top f (M\widetilde Pu,y) + \widetilde P^\top P^\top r(y).
  \end{align*}
  Since $f$ is continuous, also $F$ is continuous. From the strong monotony of $f$
  with respect to $x$ we can also conclude the strong monotony of $F$
  with respect to $u$ since there is a constant $c>0$ such that,
  for all $y\in\mathbb{R}^n$ and for all $u,\bar u\in\mathbb{R}^r$,
  \begin{align*}
    \langle F(u,y)- F(\bar u,y), u-\bar u \rangle
    &= \langle (M\widetilde P)^\top f (M\widetilde Pu,y) - (M\widetilde P)^\top f (M\widetilde P\bar u,y), u-\bar u \rangle\\
    &= \langle f (M\widetilde Pu,y) - f (M\widetilde P\bar u,y), M\widetilde Pu-M\widetilde P\bar u \rangle\\
    &\ge c \|M\widetilde Pu-M\widetilde P\bar u\|^2
  \end{align*}
  and 
  \begin{align*}
     \|u-\bar u\| &= \|((M\widetilde P)^\top M\widetilde P)^{-1} (M\widetilde P)^\top M\widetilde P(u-\bar u)\|\\
     &\le \|((M\widetilde P)^\top M\widetilde P)^{-1} (M\widetilde P)^\top\|\,\| M\widetilde P(u-\bar u)\|
  \end{align*}
  which implies
  \begin{align*}
    \langle F(u,y)- F(\bar u,y), u-\bar u \rangle
    &\ge \frac{c}{c_1} \|u-\bar u\|^2
  \end{align*}
  for $c_1:=\|((M\widetilde P)^\top M\widetilde P)^{-1} (M\widetilde P)^\top\|^2>0$ since $M$ is
  a non-zero matrix. From Lemma \ref{lem:solvability.monotone} we know that there
  is a unique continuous function $G:\mathbb{R}^n\to\mathbb{R}^r$ such that
  \begin{align*}
    F(G(y),y)=0 \quad\forall\,y\in\mathbb{R}^n.
  \end{align*}
  It means that $F(u,y)=0$ if and only if $u=G(y)$. Next, we show that the 
  function $g:\mathbb{R}^n\to\mathbb{R}^k$ defined by
  \begin{align*}
    g(y):=\widetilde PG(y)
  \end{align*}
  satisfies the equivalence \eqref{eq:equivalence}. First, we see that
  \begin{align*}
    \widetilde P^\top M^\top f (M g(y),y) + \widetilde P^\top P^\top r(y)
    = F(G(y),y) = 0.
  \end{align*}
  By construction of $\widetilde P$, we know that $\ker\,\widetilde P^\top =\ker\,P^\top$ and,
  therefore,
  \begin{align*}
    P^\top M^\top f (M g(y),y) + P^\top P^\top r(y) =0.
  \end{align*}
  Since $P$ is a projector along $\ker\,M$, we see that
  $M = MP$ and, hence,
  \begin{align*}
    M^\top f (M g(y),y) + P^\top r(y) =0.
  \end{align*}
  From here, we can directly conclude the following direction of the
  equivalence \eqref{eq:equivalence}. If $Pz = g(y)$ then $M^\top f(Mz,y) + P^\top r(y) = 0$.
  Finally, we show the opposite direction. If $M^\top f(Mz,y) + P^\top r(y) = 0$ then we again
  exploit the monotony of $f$ in order to obtain
  \begin{align*}
    0 &= \langle M^\top f(Mz,y) + P^\top r(y) - M^\top f(Mg(y),y) - P^\top r(y), z-g(y)\rangle\\
    &= \langle f(Mz,y) - f(Mg(y),y), Mz-Mg(y)\rangle 
    \ge c \| Mz-Mg(y)\|,
  \end{align*}
  that means $M(z- g(y))=0$. By assumption we have $\ker\,M= \ker\,P$ 
  and, hence, $P(z-g(y))=0$. It follows $Pz = Pg(y)=P\widetilde PG(y)=\widetilde PG(y)=g(y)$. \qed
\end{proof}

\medskip
\begin{corollary}
  \label{cor:solutionform}
  Let $M\in\mathbb{R}^{k\times m}$ be a matrix and $P\in\mathbb{R}^{k\times k}$
  be a projector along $\ker\,M$. Additionally, let
  $r:\mathbb{R}^n\to \mathbb{R}^m$ be continuous.
  Then, there is a continuous function
  $g:\mathbb{R}^n\to\mathbb{R}^k$ such that 
  \begin{align}
    \label{eq:equivalence2}
    M^\top Mz + P^\top r(y) = 0 \quad\text{if and only if}\quad Pz = g(y).
  \end{align}
\end{corollary}

\begin{proof}
  It follows directly from Lemma \ref{lem:solutionform.monotone} using the
  function $F:\mathbb{R}^m\times \mathbb{R}^n\to \mathbb{R}^m$ defined by
  \begin{align*}
    F(x,y) := x.
  \end{align*}\qed
\end{proof}

Following \cite{Brenan_1995aa} we call a function $x(t)$ the solution of a general nonlinear DAE
	\begin{equation}
		\label{eq:brenan}
		f(x',x,t) = 0
	\end{equation}
on an interval $\mathcal{I}\subset \setR$, if $x$ is continuously differentiable on $I$ and satisfies \eqref{eq:brenan} for
all $t\in \mathcal{I}$.

\begin{assumption}
  We assume solvability of \eqref{eq:brenan}, see e.g. \cite[Definition 2.2.1]{Brenan_1995aa}, and that all functions involved are sufficiently smooth.
\end{assumption}

\begin{definition}[\cite{Brenan_1995aa}]
	\label{def:diffindex}
	The minimum number of times that all or part of \eqref{eq:brenan} must be differentiated with respect to $t$ in order to determine $x'$ as a continuous function of $x$, $t$, is the index \index{differentiation index} of the DAE.
\end{definition}

\section{Generalized Circuit Elements}
\label{sec:generalized_circuit_elements}

In this section we define new classes of \index{generalized circuit elements} generalized circuit elements motivated by the classical ones, i.e., resistances, inductances and capacitances. The first inductance-like element is based on the definition in \cite{Cortes-Garcia_2019aa}. The original version was designed to represent a specific class of models but also to be minimally invasive in the sense that the proofs in \cite{Estevez-Schwarz_2000aa} could still be used. The following definition is more general and a new proof of the corresponding index results is given in Section~\ref{sec:DAE_index}. 

\begin{definition}\label{def:ind-like}
	We define an \textbf{inductance-like} element as one element described by
	\begin{align*}
	\find\left(\dxdy{}{t}\dfind(\xind, \iind, \vind, t) , \xind, \iind, \vind,t\right) = 0
	\end{align*}
  where there is at most one differentiation $\dxdy{}{t}$ needed to obtain a
	model description of the form
  \begin{align}
		\dxdy{}{t}\xind &= \xfind(\dxdy{}{t}\vind,\xind,\iind,\vind,t) \label{eq:ind.xf}\\
		\dxdy{}{t}\iind &= \gind(\xind,\iind,\vind,t) \label{eq:ind.g}
  \end{align}
  We call it a \textbf{strongly inductance-like} element if, additionally, the function
  \begin{align}
		\strongfind(\vpind,\xind,\iind,\vind,t):=&\
		\partial_\xind \gind(\xind,\iind,\vind,t)\xfind(\vpind,\xind,\iind,\vind,t)\nonumber\\
		& +\partial_\vind \gind(\xind,\iind,\vind,t) \vpind
    \label{eq:ind.strongf}
  \end{align}
  is continuous and strongly monotone with respect to $\vpind$.
\end{definition}

\begin{proposition}\label{prop:classicind}
	Linear \index{inductance} inductances defined as
	\begin{equation*}
		\vind-L\dxdy{}{t}\iind = 0\;,
	\end{equation*}
	with $L$ being positive definite, are strongly inductance-like elements.
\end{proposition}
\begin{proof}
	By inverting $L$ we obtain without the need of any differentiation a 
	model description as required in \eqref{eq:ind.g} in Definition~\ref{def:ind-like}. 
	Furthermore, $\strongfind(\vind') = L^{-1}\vpind$ is strongly monotone with respect to $\vpind$ due to $L^{-1}$ being positive 
	definite and by using Lemma~\ref{lem:posdef.monotone} in Definition~\ref{def:stronmonomxn}. \qed
\end{proof}

\begin{proposition}\label{prop:fluxind}
	Flux formulated inductances defined as
	\begin{align*}
		\vind &= \dxdy{}{t}\Phi_{\scalebox{0.5}{$\mathrm{L}$}} \;,\\
		\Phi_{\scalebox{0.5}{$\mathrm{L}$}} &= \phi(\iind, t)\;,
	\end{align*}
	with ${\partial_{\iind}} \phi(\iind, t)$ being positive definite, are strongly inductance-like elements.
\end{proposition}
\begin{proof}
	we chose $\xind = \Phi_{\scalebox{0.5}{$\mathrm{L}$}}$. 
	Then, one time differentiation of the second equation yields \( \dxdy{}{t} \xind =  {\partial_{\iind}} \phi(\iind, t)	\dxdy{}{t} \iind  +  {\partial_t}\phi(\iind, t) \) 
	and exploiting the positive definiteness we write \( \dxdy{}{t} \iind \) as in \eqref{eq:ind.g} Definition~\ref{def:ind-like}, for
	\begin{align*}
		\gind(\iind,\vind,t) := {\partial_{\iind}}\phi(\iind, t)^{-1}\dxdy{}{t} \xind - {\partial_t}\phi(\iind, t) 
		= {{\partial_{\iind}}\phi(\iind, t)^{-1}} \vind - {\partial_t}\phi(\iind, t)\;.
	\end{align*}
	Consequently, $\strongfind(\vind', \iind, t) = {\partial_{\iind}}\phi(\iind, t)^{-1} \vpind$ and $\strongfind(\vind', \iind, t)$ is strongly monotone with respect to $\vpind$. The latter follows again from ${\partial_{\iind}}\phi(\iind, t)^{-1}$ being positive 
	definite and by using Lemma~\ref{lem:posdef.monotone} in Definition~\ref{def:stronmonomxn}. \qed
\end{proof}
A more complex application of an electromagnetic element complying with this definition can be found in Section~\ref{sec:ind_like_full_ME}.

\begin{definition}\label{def:cap-like}
	We define a \textbf{capacitance-like} element as one element described by
	\begin{align*}
	\fcap\left(\dxdy{}{t}\dfcap(\xcap, \icap, \vcap, t) , \xcap, \icap, \vcap,t\right) = 0
	\end{align*}
  where there is at most one differentiation $\dxdy{}{t}$ needed to obtain a
	model description of the form
  \begin{align}
		\dxdy{}{t}\xcap &= \xfcap(\dxdy{}{t}\icap,\xcap,\icap,\vcap,t) \label{eq:cap.xf}\\
		\dxdy{}{t}\vcap &= \gcap(\xcap,\icap,\vcap,t) \label{eq:cap.g}
  \end{align}
  We call it a \textbf{strongly capacitance-like} element if, additionally, the function
  \begin{align}
		\strongfcap(\ipcap,\xcap,\icap,\vcap,t):=&\
		\partial_\xcap \gcap(\xcap,\icap,\vcap,t)\xfcap(\ipcap,\xcap,\icap,\vcap,t)\nonumber\\
		& +\partial_\icap \gcap(\xcap,\icap,\vcap,t) \ipcap
    \label{eq:cap.strongf}
  \end{align}
  is continuous and strongly monotone with respect to $\ipcap$.
\end{definition}

\begin{proposition}
\label{prop:classiccap}
Linear \index{capacitance} capacitances defined as
\begin{equation*}
	C\dxdy{}{t}\vcap-\icap = 0\;,
\end{equation*}
with $C$ being positive definite, are strongly capacitance-like elements.
\end{proposition}
\begin{proof}
	Analogous to the proof in Proposition~\ref{prop:classicind}, we exploit the fact that $C$ is positive definite and here,
	$\strongfcap(\icap') = C^{-1}\ipcap$ is shown to be strongly monote with respect to $\ipcap$ by using by using Lemma~\ref{lem:posdef.monotone} 
	and Definition~\ref{def:stronmonomxn}. \qed
\end{proof}

\begin{proposition}
	\label{prop:chargecap}
	Charge formulated capacitances defined as
	\begin{align*}
	\icap &= \dxdy{}{t}q_{\scalebox{0.5}{$\mathrm{C}$}} \;,\\
	q_{\scalebox{0.5}{$\mathrm{C}$}} &= q(\vcap, t)\;,
	\end{align*}	
	with ${\partial_{\vcap}} q(\vcap, t)$ being positive definite, are strongly capacitance-like elements.
\end{proposition}
\begin{proof}
	There proof is analogous to the one of Proposition~\ref{prop:fluxind} by setting $\xcap = q_{\scalebox{0.5}{$\mathrm{C}$}}$ and
	$\strongfcap(\ipcap,\vcap,t) = {\partial_t}q(\vcap,t)^{-1}\ipcap$. \qed
\end{proof}

\begin{definition}\label{def:res-like}
	We define a \textbf{resistance-like} element as one element described by
	\begin{align*}
	\fres\left(\dxdy{}{t}\dfres(\xres, \ires, \vres, t) , \xres, \ires, \vres,t\right) = 0
	\end{align*}
  where there is at most one differentiation $\dxdy{}{t}$ needed to obtain a
	model description of the form
  \begin{align}
		\dxdy{}{t}\xres &= \xfres(\xres,\ires,\vres,t) \label{eq:res.xf}\\
		\dxdy{}{t}\ires &= \gres(\dxdy{}{t}\vres,\xres,\ires,\vres,t) \label{eq:res.g}
  \end{align}
  We call it a \textbf{strongly resistance-like} element if, additionally, the function
  \begin{align}
		\gres(\vpres,\xres,\ires,\vres,t)
    \label{eq:res.strongf}
  \end{align}
  is continuous and strongly monotone with respect to $\vpres$.
\end{definition}

\begin{proposition}\label{prop:classicres}
	Linear \index{resistance} resistances defined as
	\begin{equation*}
	\vres-R\ires = 0\;,
	\end{equation*}
	with $R$ being positive definite, are strongly resistance-like elements.
\end{proposition}
\begin{proof}
	Here, the equation is differentiated once to obtain 
	$$\dxdy{}{t}\vres-R\dxdy{}{t}\ires = 0.$$
	Now, analogously to the proof in \ref{prop:classicind}, we exploit the positive definiteness of 
	$R$ to invert it and obtain a function $\gres(\vpres) = R^{-1}\vpres$, which is strongly monote
	with respect to $\vpres$. \qed
\end{proof}

\begin{remark}
  Definitions~\ref{def:ind-like}-\ref{def:res-like} are made for one-port elements or multi-port elements which are structurally identically for each port and do not change their structure, e.g. depending on state, time (or frequency). However, in practice an inductance-like device may turn into a capacitance-like device depending on its working point. Also, a two-port element may simply consist of an inductance and a capacitance. Those examples are not covered by our generalizations.
\end{remark}

\section{Circuit Structures and Circuit Graph Describing Matrices}
\label{sec:circuit_structures}
In this section we define the common ingredients for the analysis of \index{circuit} circuits, see e.g. \cite{Ruehli_2015aa,Estevez-Schwarz_2000aa}.

\begin{assumption}
  \label{ass:circuit}
  Let a connected circuit be given whose elements belong to the set of
  capacitance-like devices, inductance-like devices, resistance-like devices,
  voltage sources and current sources.
\end{assumption}

\medskip
We consider the element related incidence matrices $\Acap$, $\Aind$, $\Ares$, $\Avsrc$ and $\Acsrc$ whose entries $a_{ij}$ are defined by
\begin{align*}
  a_{ij} =
  \begin{cases}
    +1 & \text{if branch } j \text{ directs from node } i \\
    -1 & \text{if branch } j \text{ directs to node } i \\
    0 & \text{else}
  \end{cases}
\end{align*}
where the index $i$ refers to a node (except the mass node) and the index $j$ refers to branches of capacitance-like devices ($\Acap$), inductance-like devices ($\Aind$), 
resistance-like devices ($\Ares$), voltage sources ($\Avsrc$) and
current sources ($\Acsrc$).

\begin{remark}
  If Assumption \ref{ass:circuit} is fulfilled then the
  incidence matrix $A$ of the circuit is given by $A=[\Acap\,\Aind\,\Ares\,\Avsrc\,\Acsrc]$ and
  has full row rank (see~\cite{Kuh-Desoer}).
\end{remark}

\begin{lemma}
  \label{lem:loops.cutsets}
  Let a connected circuit be given and $\Axtype$ be the incidence matrix of all branches
  of type X. All other branches shall be collected in the incidence matrix $\Aytype$ such
  that the incidence matrix of the circuit is given by $A=[\Axtype\,\Aytype]$.
  Then,
  \begin{enumerate}
  \item
    the circuit contains no \index{loop} loops of only X-type branches if and only if $\Axtype$ has full
    column rank,
  \item
    the circuit contains no \index{cutset} cutsets of only X-type branches if and only if $\Aytype$ has full
    row rank.
  \end{enumerate}
\end{lemma}

\begin{proof}
  The incidence matrix of a subset $S$ of branches of a circuit is non-singular if and
  only if $S$ forms a spanning tree \cite{Kuh-Desoer}. From
  this we can conclude the following statements.
  \begin{enumerate}
  \item
    The circuit contains no loops of only X-type branches if and only if there is a
    spanning tree containing all X-type branches. The latter condition is equivalent to
    the condition that $\Axtype$ has full column rank.
  \item
    The circuit contains no cutsets of only X-type branches if and only if there is a
    spanning tree containing only Y-type branches. The latter condition is equivalent to
    the condition that $\Aytype$ has full row rank.
  \end{enumerate} \qed
\end{proof}

\begin{corollary}
  \label{cor:IcutsetVloop}
  Let Assumption \ref{ass:circuit} be fulfilled. Then,
  \begin{enumerate}
  \item
    the circuit contains no loops of only voltage sources if and only if $\Avsrc$
    has full column rank,
  \item
    the circuit contains no cutsets of only current sources if and only if
    $[\Acap\,\Aind\,\Ares\,\Avsrc]$  has full row rank.
  \end{enumerate}
\end{corollary}

Since loops of only voltage sources and cutsets of only current sources are
electrically forbidden, we suppose the following assumption to be fulfilled.

\begin{assumption}
  \label{ass:Vloops.Icutsets}
  The matrix $\Avsrc$ has full column rank and the matrix 
  $[\Acap\,\Aind\,\Ares\,\Avsrc]$  has full row rank.
\end{assumption}

\begin{definition}
  We call a loop of branches of a circuit a \textbf{CV-loop} if it contains only
  capacitance-like devices and voltage sources. We call a cutset of branches of a circuit
  an \textbf{LI-cutset} if it contains only inductance-like devices and current sources.
\end{definition}

\begin{corollary}
  \label{cor:LIcutsetCVloop}
  Let Assumption \ref{ass:circuit} be fulfilled. Then,
  \begin{enumerate}
  \item
    the circuit contains no CV-loops if and only if $[\Acap\,\Avsrc]$
    has full column rank,
  \item
    the circuit contains no LI-cutsets if and only if
    $[\Acap\,\Ares\,\Avsrc]$  has full row rank.
  \end{enumerate}
\end{corollary}

\section{DAE Index for Circuits with Generalized Lumped Models}
\label{sec:DAE_index}

Let Assumption \ref{ass:circuit} and Assumption \ref{ass:Vloops.Icutsets} be fulfilled.
Following the idea of the \index{modified nodal analysis} modified nodal analysis for circuits, we introduce the nodal potentials $e$ and form the circuit equations as
\begin{subequations}
\label{eq:MNAequations}
\begin{align}
  \label{eq:MNA.KCL}
  \Acap\icap + \Ares\ires + \Avsrc\ivsrc + \Aind\iind + \Acsrc\isrc &= 0,\\
  \label{eq:MNA.vsrc}
  \Atvsrc e &= \vsrc,\\
  \label{eq:MNA.ind}
  \find\left(\dxdy{}{t}\dfind(\xind, \iind, \Atind e, t) , \xind, \iind, \Atind e,t\right) &= 0,\\
  \label{eq:MNA.cap}
  \fcap\left(\dxdy{}{t}\dfcap(\xcap, \icap, \Atcap e, t) , \xcap, \icap, \Atcap e,t\right) &= 0,\\
  \label{eq:MNA.res}
  \fres\left(\dxdy{}{t}\dfres(\xres, \ires, \Atres e, t) , \xres, \ires, \Atres e,t\right) &= 0
\end{align}
\end{subequations}
with given source functions $\isrc=\isrc(t)$ for current
sources and $\vsrc=\vsrc(t)$  for voltage sources.

\begin{remark}\label{rem:extra_currents}
Please note that the currents $i_C$ and $i_R$ are variables of the system \eqref{eq:MNAequations}. This is in contrast to the traditional modified nodal analysis which is only based on simple lumped elements such that these variables can be eliminated by explicitly solving \eqref{eq:MNA.res} and \eqref{eq:MNA.cap} for the currents $i_C$ and $i_R$, respectively.
\end{remark}

\begin{theorem}
  \label{th:index1}
  Let Assumption \ref{ass:circuit} be fulfilled. Furthermore, let all resistance-like devices
  be strongly resistance-like devices.
  If the circuit has no CV-loops and no LI-cutsets
  then the differentiation index of the system \eqref{eq:MNAequations} is at most index 1.
\end{theorem}

\begin{proof}
  Let $\Qcapvsrc$ be a projector onto $\ker\,[\Acap\,\Avsrc]^\top$ and
  $\Pcapvsrc:=I-\Qcapvsrc$. It allows us to split
  \begin{align*}
    e &= \Pcapvsrc e+\Qcapvsrc e.
  \end{align*}
  For the capacitance-like devices and the voltage sources we find after
  at most one differentiation of the device equations \eqref{eq:MNA.cap} and
  \eqref{eq:MNA.vsrc} that
  \begin{align}
    \label{eq:vcapvvsrc}
    \Atcap \dxdy{}{t}e = \gcap(\xcap,\icap,\Atcap e,t)
    \quad\text{and}\quad
    \Atvsrc \dxdy{}{t}e = \dxdy{}{t}\vsrc.
  \end{align}
  It implies
  \begin{align*}
    [\Acap\,\Avsrc]
    \left(
    \begin{bmatrix}
      \Atcap\\
      \Atvsrc
    \end{bmatrix} \dxdy{}{t}e
    -
   \begin{bmatrix}
      \gcap(\xcap,\icap,\Atcap e,t)\\
      \dxdy{}{t}\vsrc
    \end{bmatrix}
    \right)
    &=0.
  \end{align*}
  Applying Corollary \ref{cor:solutionform} for $M := [\Acap\,\Avsrc]^\top$,
  $P:=\Pcapvsrc$,
  \begin{align*}
    z:=\dxdy{}{t}e,\
    y:=(\xcap,\icap,e,t),\
    f(y):=\begin{bmatrix} \gcap(\xcap,\icap,\Atcap e,t)\\ \dxdy{}{t}\vsrc(t) \end{bmatrix},\ 
    r(y):=0
  \end{align*}    
  we find a continuous function $f_1$ such that 
  \begin{align}
    \label{eq:Pcapvsrc.e}
    \Pcapvsrc \dxdy{}{t}e = f_1(\xcap,\icap,e,t).
  \end{align}
  Next we exploit the nodal equations \eqref{eq:MNA.KCL}. Multiplication by $\Qtcapvsrc$
  and one differentiation yields
  \begin{align}
    \label{eq:Qtcapvsrc.KCL}
    \Qtcapvsrc(\Ares \dxdy{}{t}\ires + \Aind \dxdy{}{t}\iind + \Acsrc \dxdy{}{t}\isrc)=0.
  \end{align}
  For the resistance-like and inductance-like devices we get after
  at most one differentiation of the device equations \eqref{eq:MNA.res} and 
  \eqref{eq:MNA.ind} that
  \begin{align}
    \label{eq:ires}
    \dxdy{}{t}\ires = \gres(\dxdy{}{t}\Atres e,\xres,\ires,\Atres e,t)
    = \gres(\Atres\Qcapvsrc\dxdy{}{t}e+\Atres\Pcapvsrc\dxdy{}{t}e,\xres,\ires,\Atres e,t)
  \end{align}
  and
  \begin{align}
    \label{eq:iind}
    \dxdy{}{t}\iind = \gind(\xind,\iind,\Atind e,t).
  \end{align}  
  Together with \eqref{eq:Pcapvsrc.e} and \eqref{eq:Qtcapvsrc.KCL}
  we obtain
  \begin{align}
    \Qtcapvsrc\left(
      \Ares \gres(\Atres\Qcapvsrc\dxdy{}{t}e+\Atres f_1(\xcap,\icap,e,t),
      \xres,\ires,\Atres e,t) \right. & \nonumber\\
      \left. + \Aind \gind(\xind,\iind,\Atind e,t) + \Acsrc \dxdy{}{t}\isrc
    \right) &=0.
    \label{eq:Qtcapvsrc} 
  \end{align}
  We choose a projector $\Presmcapvsrc$ along $\ker\,\Atres\Qcapvsrc$.
  Then, multiplication of \eqref{eq:Qtcapvsrc} by $\Ptresmcapvsrc$ yields
  \begin{align}
    \Qtcapvsrc
      \Ares \gres(\Atres\Qcapvsrc\dxdy{}{t}e+\Atres f_1(\xcap,\icap,e,t),
      \xres,\ires,\Atres e,t) & \nonumber\\
      + \Ptresmcapvsrc \Qtcapvsrc
      ( \Aind \gind(\xind,\iind,\Atind e,t) + \Acsrc \dxdy{}{t}\isrc) &=0.
      \label{eq:Ptresmcapvsrc}
  \end{align}
  It allows us to apply Lemma \ref{lem:solutionform.monotone} for
  \begin{align*}
    M:= \Atres\,\Qcapvsrc,\
    P:= \Presmcapvsrc,\
    z:=\dxdy{}{t}e,\
    y:=(\xcap,\icap,\xres,\ires,\xind,\iind,e,t),
  \end{align*}    
  and
  \begin{align*}      
    f(x,y)&:=
      \gres(x+\Atres f_1(\xcap,\icap,e,t),\xres,\ires,\Atres e,t),\\
    r(y)&:= \Qtcapvsrc( \Aind \gind(\xind,\iind,\Atind e,t) + \Acsrc \dxdy{}{t}\isrc).
  \end{align*}    
  Thus, we find a continuous function $f_2$ such that 
  \begin{align}
    \label{eq:Presmcapvsrc.e}
    \Presmcapvsrc \dxdy{}{t}e = f_2(\xcap,\icap,\xres,\ires,\xind,\iind,e,t).
  \end{align}
  Since the circuit does not contain LI-cutsets, the matrix
  $[\Acap\,\Ares\,\Avsrc]^\top$ has full column rank (see Corollary
  \ref{cor:LIcutsetCVloop}). It implies for $\Qresmcapvsrc:=I-\Presmcapvsrc$ that
  \begin{align*}
    \ker\,\Qcapvsrc = \ker\,\Atres \Qcapvsrc = \ker\,\Presmcapvsrc = \im\,\Qresmcapvsrc
  \end{align*}
  and, therefore, $\Qcapvsrc=\Qcapvsrc\Presmcapvsrc$. Consequently,
  \begin{align}
    \label{eq:Qcapvsrc.e}
    \Qcapvsrc \dxdy{}{t}e = \Qcapvsrc f_2(\xcap,\icap,\xres,\ires,\xind,\iind,e,t).
  \end{align}
  Regarding \eqref{eq:Pcapvsrc.e}, \eqref{eq:Qcapvsrc.e} and \eqref{eq:ires},
  we find continuous functions $f_3$ and $f_4$ such that
  \begin{align}
    \label{eq:e.ires}
    \dxdy{}{t}e = f_3(\xcap,\icap,\xres,\ires,\xind,\iind,e,t)
    \quad\text{and}\quad
    \dxdy{}{t}\ires = f_4(\xcap,\icap,\xres,\ires,\xind,\iind,e,t).
  \end{align}
  Using again \eqref{eq:MNA.KCL}, we get
  \begin{align*}
    [\Acap\,\Avsrc] \begin{bmatrix} \dxdy{}{t}\icap \\[0.5ex] \dxdy{}{t}\ivsrc \end{bmatrix}
    + \Ares\dxdy{}{t}\ires + \Aind\dxdy{}{t}\iind + \Acsrc\dxdy{}{t}\isrc &= 0.
  \end{align*}
  Together with \eqref{eq:e.ires} and \eqref{eq:iind} we have
  \begin{align}
    \label{eq:icap.ivsrc.KCL}
    [\Acap\,\Avsrc] \begin{bmatrix} \dxdy{}{t}\icap \\[0.5ex] \dxdy{}{t}\ivsrc \end{bmatrix}
    + \Ares f_4(\xcap,\icap,\xres,\ires,\xind,\iind,e,t)
    + \Aind \gind(\xind,\iind,\Atind e,t) + \Acsrc\dxdy{}{t}\isrc &= 0.
  \end{align}
  Since the circuit does not contain CV-loops, the matrix $[\Acap\,\Avsrc]$ has full
  column rank and, hence, $\ker\,[\Acap\,\Avsrc]={0}$. Multiplying \eqref{eq:icap.ivsrc.KCL}
  by $[\Acap\,\Avsrc]^\top$ allows us to apply Corollary \ref{cor:solutionform} for
  \begin{align*}
    M:=[\Acap\,\Avsrc],\ P:=I,\
    z:= \begin{bmatrix} \dxdy{}{t}\icap \\[0.5ex] \dxdy{}{t}\ivsrc \end{bmatrix},\
    y:= (\xcap,\icap,\xres,\ires,\xind,\iind,e,t)
  \end{align*}
  and
  \begin{align*}      
    f(y):=
      \Ares f_4(\xcap,\icap,\xres,\ires,\xind,\iind,e,t)
     + \Aind \gind(\xind,\iind,\Atind e,t) + \Acsrc\dxdy{}{t}\isrc(t).
  \end{align*}    
  Consequently, we find a continuous function $f_5$ such that
  \begin{align}
    \label{eq:icap.ivsrc}
    \begin{bmatrix} \dxdy{}{t}\icap \\[0.5ex] \dxdy{}{t}\ivsrc \end{bmatrix}
    = f_5(\xcap,\icap,\xres,\ires,\xind,\iind,e,t).
  \end{align}
  Finally, we obtain from \eqref{eq:ind.xf} and \eqref{eq:e.ires} that
  \begin{align}
    \label{eq:xind.index1}
    \dxdy{}{t}\xind &= \xfind(\Atind f_3(\xcap,\icap,\xres,\ires,\xind,\iind,e,t),\xind,\iind,\vind,t)
  \end{align}
  and from \eqref{eq:cap.xf} and \eqref{eq:icap.ivsrc} that
  \begin{align}
    \label{eq:xcap.index1}
    \dxdy{}{t}\xcap &= \xfcap([I\,0]f_5(\xcap,\icap,\xres,\ires,\xind,\iind,e,t),
      \xcap,\icap,\vcap,t).
  \end{align}
  Consequently, the equations \eqref{eq:e.ires}, \eqref{eq:iind}, \eqref{eq:icap.ivsrc} and
  \eqref{eq:xind.index1}, \eqref{eq:xcap.index1}, \eqref{eq:res.xf} represent an explicit
  ordinary differential equation system. That means the differentiation index of the
  circuit system \eqref{eq:MNAequations} is at most 1. \qed
\end{proof}

\begin{theorem}
  \label{th:index2}
  Let Assumption \ref{ass:circuit} and Assumption \ref{ass:Vloops.Icutsets} be fulfilled.
  Furthermore, let all resistance-like devices
  be strongly resistance-like devices. Additionally, let all inductance-like devices belonging
  to LI-cutsets be strongly inductance-like devices and 
  all capacitance-like devices belonging to CV-loops be strongly capacitance-like devices.
  Then, the differentiation index of the system \eqref{eq:MNAequations} is at most index 2.
\end{theorem}

\begin{proof}
  First, we follow the proof of Theorem \ref{th:index1} and derive the equations
  \eqref{eq:vcapvvsrc}-\eqref{eq:Presmcapvsrc.e}.
  Then, multiplication of \eqref{eq:Qtcapvsrc} by $\Qtresmcapvsrc:=I-\Ptresmcapvsrc$ yields
  \begin{align}
      \Qtresmcapvsrc \Qtcapvsrc
      ( \Aind \gind(\xind,\iind,\Atind e,t) + \Acsrc \dxdy{}{t}\isrc) &=0.
      \label{eq:Qtresmcapvsrc}
  \end{align}
  Differentiating \eqref{eq:Qtresmcapvsrc} once again, we obtain
  \begin{align*}
      \Qtresmcapvsrc \Qtcapvsrc \left(
      \Aind \partial_\xind\gind(\xind,\iind,\Atind e,t) \dxdy{}{t}\xind 
      + \Aind \partial_\iind\gind(\xind,\iind,\Atind e,t) \dxdy{}{t}\iind \right. & \\
      + \Aind \partial_\vind\gind(\xind,\iind,\Atind e,t) \dxdy{}{t}\Atind e 
      + \Aind \partial_t\gind(\xind,\iind,\Atind e,t) 
      \left.+ \Acsrc \dxdy{^2}{t^2}\isrc \right) &=0.
  \end{align*}
  Next, we plug in \eqref{eq:ind.xf} and \eqref{eq:iind}. Hence,
  \begin{align*}
      \Qtresmcapvsrc \Qtcapvsrc \left(
      \Aind \partial_\xind\gind(\xind,\iind,\Atind e,t) \xfind(\dxdy{}{t}\Atind e,\xind,\iind,\vind,t)
      + \Aind \partial_\vind\gind(\xind,\iind,\Atind e,t) \dxdy{}{t}\Atind e 
      \right. & \\
      + \Aind \partial_\iind\gind(\xind,\iind,\Atind e,t) \gind(\xind,\iind,\Atind e,t)
      + \Aind \partial_t\gind(\xind,\iind,\Atind e,t) 
      \left.+ \Acsrc \dxdy{^2}{t^2}\isrc \right) =&\ 0.
  \end{align*}
  Using \eqref{eq:ind.strongf}, we see that
  \begin{align}
      \Qtresmcapvsrc \Qtcapvsrc \left(
      \Aind \strongfind(\dxdy{}{t}\Atind e,\xind,\iind,\Atind e,t)
      + \Aind \partial_t\gind(\xind,\iind,\Atind e,t) 
      \right. & \nonumber \\
      + \Aind \partial_\iind\gind(\xind,\iind,\Atind e,t) \gind(\xind,\iind,\Atind e,t)
      \left.+ \Acsrc \dxdy{^2}{t^2}\isrc \right) =&\ 0.
      \label{eq:dtQtresmcapvsrc}
  \end{align}
  Regarding \eqref{eq:Presmcapvsrc.e} and \eqref{eq:Pcapvsrc.e}, we can 
  split
  \begin{align*}
    \dxdy{}{t}\Atind e &= \Atind\Qcapvsrc\Qresmcapvsrc \dxdy{}{t} e
    + \Atind\Qcapvsrc\Presmcapvsrc \dxdy{}{t} e
    + \Atind\Pcapvsrc \dxdy{}{t} e\\
    &= \Atind\Qcapvsrc\Qresmcapvsrc \dxdy{}{t}e
    + \Atind\Qcapvsrc f_2(\xcap,\icap,\xres,\ires,\xind,\iind,e,t)
    + \Atind f_1(\xcap,\icap,e,t).
  \end{align*}
  We choose a projector $\PLIcut$ along $\ker \Atind\Qcapvsrc\Qresmcapvsrc$.
  Since the circuit does not contain I-cutsets, the matrix
  $[\Acap\,\Ares\,\Avsrc\,\Aind]^\top$ has full column rank (see Corollary
  \ref{cor:IcutsetVloop}). It implies, for $\QLIcut:=I-\PLIcut$, that
  \begin{align*}
    \ker\,\Qcapvsrc\Qresmcapvsrc = \ker\,\Atind\Qcapvsrc\Qresmcapvsrc
    = \ker\,\PLIcut = \im\,\QLIcut
  \end{align*}
  and, therefore, $\Qcapvsrc\Qresmcapvsrc =\Qcapvsrc\Qresmcapvsrc \PLIcut$
  as well as $\Qtresmcapvsrc\Qtcapvsrc =\PtLIcut\Qtresmcapvsrc\Qtcapvsrc$.
  Consequently, we can apply Lemma \ref{lem:solutionform.monotone} 
  onto \eqref{eq:dtQtresmcapvsrc} with
  \begin{align*}
    M:= \Atind\Qcapvsrc\Qresmcapvsrc,\
    P:=\PLIcut,\
    z:=\dxdy{}{t}e,\
    y:=(\xcap,\icap,\xres,\ires,\xind,\iind,e,t),
  \end{align*}    
  and
  \begin{align*}      
    f(x,y):=&\
      \strongfind(x+\Atind\Qcapvsrc f_2(\xcap,\icap,\xres,\ires,\xind,\iind,e,t)
    + \Atind f_1(\xcap,\icap,e,t),\xind,\iind,\Atind e,t),\\
    r(y):=&\ \Qtresmcapvsrc \Qtcapvsrc \Aind \partial_\iind\gind(\xind,\iind,\Atind e,t)
                \gind(\xind,\iind,\Atind e,t) \\
     &{}+\Qtresmcapvsrc \Qtcapvsrc \left(
      \Aind \partial_t\gind(\xind,\iind,\Atind e,t)
      + \Acsrc \dxdy{^2}{t^2}\isrc(t) \right).
  \end{align*}    
  Thus, we find a continuous function $f_6$ such that 
  \begin{align}
    \label{eq:PLIcut.e}
    \PLIcut \dxdy{}{t}e = f_6(\xcap,\icap,\xres,\ires,\xind,\iind,e,t).
  \end{align}
  implying
  \begin{align}
    \label{eq:Qresmcapvsrc.e}
    \Qcapvsrc\Qresmcapvsrc \dxdy{}{t}e =
    \Qcapvsrc\Qresmcapvsrc f_6(\xcap,\icap,\xres,\ires,\xind,\iind,e,t).
  \end{align}
  Regarding \eqref{eq:Presmcapvsrc.e} and \eqref{eq:Pcapvsrc.e} again, we obtain
  \begin{align}
    \label{eq:e.index2}
    \dxdy{}{t} e &= f_7(\xcap,\icap,\xres,\ires,\xind,\iind,e,t)
  \end{align}
  for
  \begin{align*}
    f_7(\xcap,\icap,\xres,\ires,\xind,\iind,e,t) :=&\
    \Qcapvsrc\Qresmcapvsrc f_6(\xcap,\icap,\xres,\ires,\xind,\iind,e,t)\\
    &{}+ \Qcapvsrc f_2(\xcap,\icap,\xres,\ires,\xind,\iind,e,t)
    + f_1(\xcap,\icap,e,t).
  \end{align*}
  Regarding \eqref{eq:ires}, we get a continuous function $f_7$ such that
  \begin{align}
    \label{eq:e.ires.index2}
    \dxdy{}{t}\ires = f_7(\xcap,\icap,\xres,\ires,\xind,\iind,e,t).
  \end{align}
  Using again \eqref{eq:MNA.KCL}, we get
  \begin{align*}
    [\Acap\,\Avsrc] \begin{bmatrix} \dxdy{}{t}\icap \\[0.5ex] \dxdy{}{t}\ivsrc \end{bmatrix}
    + \Ares\dxdy{}{t}\ires + \Aind\dxdy{}{t}\iind + \Acsrc\dxdy{}{t}\isrc &= 0.
  \end{align*}
  Together with \eqref{eq:e.ires.index2} and \eqref{eq:iind} we have
  \begin{align}
    \label{eq:icap.ivsrc.KCL.index2}
    [\Acap\,\Avsrc] \begin{bmatrix} \dxdy{}{t}\icap \\[0.5ex] \dxdy{}{t}\ivsrc \end{bmatrix}
    + \Ares f_7(\xcap,\icap,\xres,\ires,\xind,\iind,e,t)
    + \Aind \gind(\xind,\iind,\Atind e,t) + \Acsrc\dxdy{}{t}\isrc &= 0.
  \end{align}
  We choose a projector $\PCVloop$ along $\ker\,[\Acap\,\Avsrc]$.
  Multiplying \eqref{eq:icap.ivsrc.KCL.index2}
  by $[\Acap\,\Avsrc]^\top$ allows us to apply Corollary \ref{cor:solutionform} for
  \begin{align*}
    M:=[\Acap\,\Avsrc],\ P:=\PCVloop,\
    z:= \begin{bmatrix} \dxdy{}{t}\icap \\[0.5ex] \dxdy{}{t}\ivsrc \end{bmatrix},\
    y:= (\xcap,\icap,\xres,\ires,\xind,\iind,e,t)
  \end{align*}
  and
  \begin{align*}      
    f(y):=
      \Ares f_7(\xcap,\icap,\xres,\ires,\xind,\iind,e,t)
     + \Aind \gind(\xind,\iind,\Atind e,t) + \Acsrc\dxdy{}{t}\isrc(t).
  \end{align*}    
  Consequently, we find a continuous function $f_8$ such that
  \begin{align}
    \label{eq:icap.ivsrc.index2}
    \PCVloop\begin{bmatrix} \dxdy{}{t}\icap \\[0.5ex] \dxdy{}{t}\ivsrc \end{bmatrix}
    = f_8(\xcap,\icap,\xres,\ires,\xind,\iind,e,t).
  \end{align}
  Rewriting \eqref{eq:vcapvvsrc} as equation system in column form and
  multiplication by $\QtCVloop$ yields
  \begin{align*}
    \QtCVloop \begin{bmatrix}
      \gcap(\xcap,\icap,\Atcap e,t)\\
      \dxdy{}{t}\vsrc
    \end{bmatrix}
    =0.
  \end{align*}
  Differentiating this equation and regarding \eqref{eq:cap.xf},
  \eqref{eq:e.index2} as well as \eqref{eq:cap.strongf}, we obtain
  \begin{align}
    \label{eq:vcapvvsrc.index2}
    \QtCVloop
    \begin{bmatrix}
      \dxdy{}{t}\gcap(\xcap,\icap,\Atcap e,t)\\
      \dxdy{^2}{t^2}\vsrc
    \end{bmatrix}
    =0
  \end{align}
  with
  \begin{align*}
    \lefteqn{\hspace{-5ex}\dxdy{}{t}\gcap(\xcap,\icap,\Atcap e,t)}& \\
    =&\ \partial_\xcap \gcap(\xcap,\icap,\Atcap e,t) \xfcap(\dxdy{}{t}\icap,\xcap,\icap,\Atcap e,t)
    +\partial_\icap \gcap(\xcap,\icap,\Atcap e,t) \dxdy{}{t}\icap \\
    &{}
    + \partial_\vcap \gcap(\xcap,\icap,\Atcap e,t)
    \Atcap f_7(\xcap,\icap,\xres,\ires,\xind,\iind,e,t)
    + \partial_t \gcap(\xcap,\icap,\Atcap e,t)\\
    =&\ \strongfcap(\dxdy{}{t}\icap,\xcap,\icap,\Atcap e,t) \\
    &{}
    + \partial_\vcap \gcap(\xcap,\icap,\Atcap e,t)
    \Atcap f_7(\xcap,\icap,\xres,\ires,\xind,\iind,e,t)
    + \partial_t \gcap(\xcap,\icap,\Atcap e,t).
  \end{align*}
  Using \eqref{eq:icap.ivsrc.index2}, we can split
  \begin{align*}
    \dxdy{}{t}\icap &= 
    \begin{bmatrix}
      I&0
    \end{bmatrix}
    \begin{bmatrix}
      \dxdy{}{t}\icap \\[0.5ex] \dxdy{}{t}\ivsrc
    \end{bmatrix}
    =
    \begin{bmatrix}
      I&0
    \end{bmatrix}
    \QCVloop
    \begin{bmatrix}
      \dxdy{}{t}\icap \\[0.5ex] \dxdy{}{t}\ivsrc
    \end{bmatrix}
    +
    \begin{bmatrix}
      I&0
    \end{bmatrix}
    \PCVloop
    \begin{bmatrix}
      \dxdy{}{t}\icap \\[0.5ex] \dxdy{}{t}\ivsrc
    \end{bmatrix}\\
    &=
    \begin{bmatrix}
      I&0
    \end{bmatrix}
    \QCVloop
    \begin{bmatrix}
      \dxdy{}{t}\icap \\[0.5ex] \dxdy{}{t}\ivsrc
    \end{bmatrix}
    +
    \begin{bmatrix}
      I&0
    \end{bmatrix}
    f_8(\xcap,\icap,\xres,\ires,\xind,\iind,e,t).
  \end{align*}   
  Since the circuit has no V-loop, the matrix $\Avsrc$ has full column rank, see
  Corollary \ref{cor:IcutsetVloop}. It implies
  \begin{align*}
    \ker \begin{bmatrix} I&0 \end{bmatrix} \QCVloop
    = \ker \begin{bmatrix} I&0\\ \Acap & \Avsrc \end{bmatrix} \QCVloop
    = \ker \QCVloop.
  \end{align*}
  Rewriting \eqref{eq:vcapvvsrc.index2} as
  \begin{align*}
    \QtCVloop
    \begin{bmatrix}
      I\\0
    \end{bmatrix}
    \dxdy{}{t}\gcap(\xcap,\icap,\Atcap e,t)
    +
    \QtCVloop
    \begin{bmatrix}
      0\\I
    \end{bmatrix}
      \dxdy{^2}{t^2}\vsrc
    =0.
  \end{align*}
  allows us to apply Lemma \ref{lem:solutionform.monotone} with
  \begin{align*}
    M:=\begin{bmatrix} I&0 \end{bmatrix}\QCVloop,\
    P:= \QCVloop,\
    z:=\begin{bmatrix}
      \dxdy{}{t}\icap \\[0.5ex] \dxdy{}{t}\ivsrc
    \end{bmatrix},\
    y:=(\xcap,\icap,\xres,\ires,\xind,\iind,e,t)
  \end{align*}    
  and
  \begin{align*}      
    f(x,y):=&\
    \strongfcap(x+\begin{bmatrix} I&0 \end{bmatrix}
    f_8(\xcap,\icap,\xres,\ires,\xind,\iind,e,t),\xcap,\icap,\Atcap e,t) \\
    &{}
    + \partial_\vcap \gcap(\xcap,\icap,\Atcap e,t)
    \Atcap f_7(\xcap,\icap,\xres,\ires,\xind,\iind,e,t)
    + \partial_t \gcap(\xcap,\icap,\Atcap e,t),\\
    r(y):=&  \begin{bmatrix} 0\\I \end{bmatrix}
    \dxdy{^2}{t^2}\vsrc(t).
  \end{align*}
  It means that we find a continuous function $f_9$ such that
  \begin{align*}
    \QCVloop
    \begin{bmatrix}
      \dxdy{}{t}\icap \\[0.5ex] \dxdy{}{t}\ivsrc
    \end{bmatrix}
    = f_9(\xcap,\icap,\xres,\ires,\xind,\iind,e,t).
  \end{align*}
  Combining it with \eqref{eq:icap.ivsrc.index2} we get
  \begin{align}
    \label{eq:icap.ivsrc.index2.final} 
    \begin{bmatrix}
      \dxdy{}{t}\icap \\[0.5ex] \dxdy{}{t}\ivsrc
    \end{bmatrix}
    = f_8(\xcap,\icap,\xres,\ires,\xind,\iind,e,t)
    + f_9(\xcap,\icap,\xres,\ires,\xind,\iind,e,t).
  \end{align}
  Finally, we obtain from \eqref{eq:ind.xf} and \eqref{eq:e.index2} that
  \begin{align}
    \label{eq:xind.index2}
    \dxdy{}{t}\xind &= \xfind(\Atind f_7(\xcap,\icap,\xres,\ires,\xind,\iind,e,t),\xind,\iind,\vind,t)
  \end{align}
  and from \eqref{eq:cap.xf} and \eqref{eq:icap.ivsrc.index2.final} that
  \begin{align}
    \label{eq:xcap.index2}
    \dxdy{}{t}\xcap &= \xfcap([I\,0](f_8+f_9)(\xcap,\icap,\xres,\ires,\xind,\iind,e,t),
      \xcap,\icap,\vcap,t).
  \end{align}
  Consequently, the equations \eqref{eq:e.index2}, \eqref{eq:iind},
  \eqref{eq:icap.ivsrc.index2.final} and
  \eqref{eq:xind.index2}, \eqref{eq:xcap.index2}, \eqref{eq:res.xf} represent an explicit
  ordinary differential equation system. That means the differentiation index of the
  circuit system \eqref{eq:MNAequations} is at most 2. \qed
\end{proof}

Theorems~\ref{th:index1} and \ref{th:index2} contain the results of \cite{Estevez-Schwarz_2000aa} in the case of circuits that only contain simple lumped elements in either traditional, i.e., Prop.~\ref{prop:classicind}, \ref{prop:classicres} and \ref{prop:classiccap}, or flux/charge formulation, i.e. Prop.~\ref{prop:fluxind} and
\ref{prop:chargecap}. Some minor differences arise due to Remark~\ref{rem:extra_currents}, e.g., loops of capacitances lead to index-2 systems since the corresponding current $i_C$ is not eliminated from the system \eqref{eq:MNAequations}. Similarly, results for many refined models, for example when considering \cite{Tsukerman_2002aa,Bartel_2011aa,Cortes-Garcia_2019aa} as inductance-like elements, are included in Theorems~\ref{th:index1} and \ref{th:index2}. The next section discusses a few challenging examples.

\section{Refined models}\label{sec:refined}
We present examples for refined models based on PDEs describing electromagnetic fields, that
are coupled to the circuit system of DAEs and can be categorized with the generalized elements 
of Section~\ref{sec:generalized_circuit_elements}.

All models appearing in this section arise from \index{Maxwell's equations} Maxwell's equations \cite{Maxwell_1864aa,Jackson_1998aa}.
Those can be written in differential form for a system at rest as
\begin{subequations}\label{eq:maxwell}
	\begin{align}
	\nabla\times\Efield &= -{\partial_t}\Bfield\;,\label{eq:afaraday_lenz_diff}\\
	\nabla\times\Hfield &= {\partial_t}{\Dfield} +\Jfield\;,\label{eq:ampere_maxwell_diff}\\
	\nabla\cdot\Dfield  &= \rhofield\;,\label{eq:gauss_diff} \\
	\nabla\cdot\Bfield  &= 0\;,\label{eq:no_mag_mono_diff}
	\end{align}
\end{subequations}
where 
$\Efield$ is the electric field strength, 
$\Bfield$ the magnetic flux density,
$\Hfield$ the magnetic field strength,
$\Dfield$ the electric flux density and
$\Jfield$ the electric current density. All these quantities are vector fields $\Omega\times\mathcal{I}\rightarrow\setR^3$
defined  in a domain $\Omega\subset\setR^3$ and time interval $\mathcal{I}\subset \setR$. The electric charge
density $\rhofield$ is a scalar field $\Omega\times\mathcal{I}\rightarrow\setR$. 

The field quantities are related to each other through the material equations
\begin{align}\label{eq:material1}
	\Dfield = \varepsilon\Efield\;, && \Jfield_{\scalebox{0.5}{$\mathrm{c}$}} = \sigma \Efield\;, && \Hfield = \mu \Bfield\;,
\end{align}
where 
$\varepsilon$ is the electric permittivity, 
$\sigma$ the electric conductivity and
$\mu$ the magnetic permeability. They are rank-2 tensor fields $\Omega\rightarrow\setR^{3\times3}$. The current density 
in \eqref{eq:ampere_maxwell_diff} can be divided into the conduction current density $\Jfield_{\scalebox{0.5}{$\mathrm{c}$}}$ of 
\eqref{eq:material1} and the source current density $\Jfield_{\scalebox{0.5}{$\mathrm{s}$}}$
\begin{equation}
\Jfield = \Jfield_{\scalebox{0.5}{$\mathrm{c}$}} + \Jfield_{\scalebox{0.5}{$\mathrm{s}$}}\;.
\end{equation}
The inverse of the material relations in \eqref{eq:material1} is defined through the 
electric resistivity 
$\rho:\Omega\rightarrow\setR^{3\times3}$ and the magnetic reluctivity $\nu:\Omega\rightarrow\setR^{3\times3}$ such that
\begin{align}\label{eq:material2}
\Efield = \rho \Jfield_{\scalebox{0.5}{$\mathrm{c}$}} \;, && \Bfield = \nu \Hfield\;.
\end{align}

\begin{assumption}[\cite{Cortes-Garcia_2018aa}]\label{ass:materials}
	We divide the space domain $\Omega$ into three disjoint subdomains $\Omega_{\scalebox{0.5}{$\mathrm{c}$}}$ (the conducting domain),
	$\Omega_{\scalebox{0.5}{$\mathrm{s}$}}$ (the source domain) and $\Omega_{\scalebox{0.5}{0}}$ (the excitation-free domain) such that
	\begin{itemize}
		\item the material tensors $\varepsilon$, $ \mu$ and $\nu$ are positive definite on the whole subdomain $\Omega$.
		\item the material tensors $\rho$ and $\sigma$ are positive definite in  $\Omega_{\scalebox{0.5}{$\mathrm{c}$}}$ and zero everywhere else.
		\item the source current density is only nonzero in $\Omega_{\scalebox{0.5}{$\mathrm{s}$}}$.
	\end{itemize}
\end{assumption}

In order to simulate Maxwell's equations and its approximations, often potentials are defined, that allow to rewrite the equations as
systems of PDEs that can be resolved. For the examples that are presented next, the magnetic vector potential 
$\Afield:\Omega\times\mathcal{I}\rightarrow\setR^3$ and the electric scalar potential $\phi:\Omega\times\mathcal{I}\rightarrow\setR$
are relevant. They are defined such that
\begin{align}\label{eq:potentials}
	\Bfield = \nabla\times\Afield && \text{and} && \Efield = -\partial_t\Afield - \nabla\phi\;.
\end{align}

Following the \textit{finite integration technique} (FIT), originally introduced in 1977 by Thomas Weiland 
\cite{Weiland_1977aa}, the discrete version of \eqref{eq:maxwell} is obtained as Maxwell's grid equations \cite{Schuhmann_2001aa}
\begin{equation}
	\curlfit e = -\dxdy{}{t}b \qquad \curldfit h = \dxdy{}{t}d + j \qquad 
	\divdfit d = q \qquad \divfit b = 0\;,
\end{equation}
here $\curlfit$, $\curldfit = \curlfit^{\top}$(see \cite{Schuhmann_2001aa}) and $\divfit$, $\divdfit$ are the discrete curl, dual curl, divergence 
and dual divergence operators, respectively. The discrete field vectors $e$, $b$, $h$, $d$, $j$ and $q$ are 
integrated quantities over points, edges, facets and volumes of two dual grids.  
Also, the material relations \eqref{eq:material1} and \eqref{eq:material2} can be formulated through the material matrices $M_{\star}$ as
\begin{equation}
	d = \Meps e \qquad j_{\mathrm{c}} = \Msigma e \qquad h = \Mmu b \qquad e = M_{\rho} j_{\mathrm{c}} \qquad b = \Mnu h\;.
\end{equation}
Analogous to the continuous case, discrete potentials can be defined, which lead to the relation
\begin{align}\label{eq:descr_def}
b = \curlfit a \qquad e = -\dxdy{}{t}a - G \bar{\Phi}\;,
\end{align}
where $a$ and $\bar{\Phi}$ are the discrete magnetic vector potential and electric scalar potential, respectively and 
$G=-\divdfit^{\top}$ (see \cite{Schuhmann_2001aa}) is the discrete gradient operator.

\begin{assumption}\label{ass:boundary}
	The boundary of the domain $\Gamma=\partial\Omega$ is divided into three disjoint sets $\Gamma_{\scalebox{0.5}{$\mathrm{neu}$,0}}$, $\Gamma_{\scalebox{0.5}{$\mathrm{dir}$,0}}$ 
	and $\Gamma_{\scalebox{0.5}{$\mathrm{s}$}}$, with
	$$\Gamma = \Gamma_{\scalebox{0.5}{$\mathrm{neu}$,0}}\cup\Gamma_{\scalebox{0.5}{$\mathrm{dir}$,0}}\cup\Gamma_{\scalebox{0.5}{$\mathrm{s}$}}\;.$$
	Here, $\Gamma_{\scalebox{0.5}{$\mathrm{neu}$,0}}$ and $\Gamma_{\scalebox{0.5}{$\mathrm{dir}$,0}}$ 
	are the parts where homogeneous
	Neumann and Dirichlet boundary conditions are imposed and $\Gamma_{\scalebox{0.5}{$\mathrm{s}$}}$ where 
	the field equation is excited. 
\end{assumption}
In case of a device described by Maxwell's equations and coupled to a circuit through boundary conditions, $\Gamma_{\scalebox{0.5}{$\mathrm{s}$}}$ 
represents the area where the device is connected to the surrounding network.

\begin{assumption}\label{ass:bcinmatrices}
	We assume that at least the homogeneous Dirichlet boundary conditions of $\Gamma_{\scalebox{0.5}{$\mathrm{dir}$,0}}$ are 
	already incorporated into the discrete operator matrices, such that the gradient
	operator matrix $G = -\divdfit^{\top}$ has full column rank.
\end{assumption}
This is a standard assumption and has already been shown and used e.g. in \cite{Baumanns_2012ab,Cortes-Garcia_2018aa}.

\begin{remark}
	Both material as well as operator matrices with similar properties are also obtained with a finite element (FE) discretization
	of the partial differential equations obtained from Maxwell's equations, whenever appropriate basis and test
	functions are used, that fulfil the discrete de Rham sequence \cite{Bossavit_1998aa,Cortes-Garcia_2019aa}. Therefore, the subsequent analysis of
	the discretized systems is also valid for FE discretizations.
\end{remark}

\newcommand{\Qs}{\ensuremath{Q_{\scalebox{0.5}{$\mathrm{s}$}}}}
\newcommand{\Ys}{\ensuremath{Y_{\scalebox{0.5}{$\mathrm{s}$}}}}

\subsection{Inductance-like element}
\label{sec:ind_like_full_ME}
In the following we give an example of an electromagnetic (EM) device, with its formulation taken from \cite{Baumanns_2013aa}, based upon full wave Maxwell's equation, 
that fits the form of a strong inductance-like element.

In the absence of source terms and Neumann boundary conditions, i.e.,\newline 
\( \Omega_{\scalebox{0.5}{$\mathrm{s}$}}, \Gamma_{\scalebox{0.5}{$\mathrm{neu}$,0}} = \emptyset \), one possibility to rewrite
 Maxwell's equations in terms of potentials is given by the following second order PDE system (see \cite{Baumanns_2012ab})
\begin{subequations}\label{eq:pot_form}
	\begin{align}
	{\varepsilon} \nabla {\partial_t} \varphi + {\zeta} \nabla \left[{\xi}\nabla \cdot \left({\zeta}\vec{A} \right) \right]&= 0 && \text{ in }\Omega\;,\label{eq:pot_gauge}\\
	\nabla\times ({\nu} \nabla\times\vec{A})
	+{\partial_t}\left[{\varepsilon}\left(\nabla\varphi+{\partial_t}\vec{A}\right)\right] + {\sigma} \left(\nabla \varphi + {\partial_t}\vec{A}\right) &= 0 && \text{ in }\Omega\;,
	\label{eq:pot_MA}
	\end{align}
\end{subequations}
where \(\zeta \) and \(\xi \) are artificial material tensors whose choice is discussed for example in \cite{clemens2002regularization} and \cite{clemens2005large}.
We refer to system \eqref{eq:pot_form} as the \( \vec{A}- \varphi \) \textit{formulation} which makes use of a \textit{grad-type Lorenz gauge condition} in order to avoid ambiguity of the potentials, see \cite{Baumanns_2012ab} \cite{clemens2002regularization}.
Let  \( \vind \) and \( \iind \) be the time-dependent branch voltages and currents of the element, respectively. 
With Assumption~\ref{ass:boundary} given, we complete \eqref{eq:pot_form} with the boundary conditions
\begin{subequations}\label{eq:pot_bc}
	\begin{align}
	\nabla\times \Afield &= 0 &&  \text{ in }\Gamma_{\scalebox{0.5}{$\mathrm{dir}$,0}}\;,\label{eq:pot_bc.a}\\
	\phi &= 0 &&  \text{ in }\Gamma_{\scalebox{0.5}{$\mathrm{dir}$,0}}\;,\label{eq:pot_bc.b}\\
	\phi &= \vind &&  \text{ in }\Gamma_{\scalebox{0.5}{$\mathrm{s}$}}\;.\label{eq:pot_bc.c}
	\end{align}
\end{subequations}
The branch currents \( \iind \) shall comply with the model
\begin{align}\label{eq:pot_cce}
\int_{\Gamma_{\scalebox{0.5}{$\mathrm{s}$}}}  \nabla \times \left( \nu \nabla \times \Afield \right) \cdot\mathrm{d}\vec{S} = \iind\;.
\end{align}

In order to apply the method of lines, we spatially discretize the system
 \eqref{eq:pot_form} using e.g. the finite integration technique.
Since most of the required matrices and quantities were already introduced in this section's 
preliminaries, we proceed with the circuit coupling which is archived via the boundaries only 
(\( \Omega_{\scalebox{0.5}{$\mathrm{s}$}} = \emptyset \)).

Given Assumption~\ref{ass:bcinmatrices}, the homogeneous Dirichlet boundaries \eqref{eq:pot_bc.a} and 
\eqref{eq:pot_bc.b} are already incorporated into the discrete operator matrices, e.g. \( \gradfit \) or 
\( \curldfit \). 
To incorporate the inhomogeneous Dirichlet boundary conditions, we split \( \bar{\Phi} \) into \( \Phi_{\scalebox{0.5}{$\mathrm{s}$}} \) and \( \Phi \), belonging to the degrees of freedom in \( \Gamma_{\scalebox{0.5}{$\mathrm{s}$}} \) and the rest, as follows
\begin{align}\label{eq:phi_split}
	\bar{\Phi} = {Q}_{\scalebox{0.5}{$\mathrm{s}$}}\Phi + {P}_{\scalebox{0.5}{$\mathrm{s}$}} \Phi_{\scalebox{0.5}{$\mathrm{s}$}}\;, 
\end{align}
with basis matrices \( {Q}_{\scalebox{0.5}{$\mathrm{s}$}} \) and \( {P}_{\scalebox{0.5}{$\mathrm{s}$}} \) of full column rank.
The boundary voltage excitation \eqref{eq:pot_bc.c} is then obtained by setting 
\( \Phi_{\scalebox{0.5}{$\mathrm{s}$}} = \Lambda_{\scalebox{0.5}{$\mathrm{s}$}}\vind \) 
with the element's terminal to $\Gamma_{\scalebox{0.5}{$\mathrm{s}$}}^{(j)}$'s degrees of freedom mapping
\begin{align*}
(\Lambda_{\scalebox{0.5}{$\mathrm{s}$}})_{ij}=
\begin{cases}
1,\quad \text{ if }(\Phi_{\scalebox{0.5}{$\mathrm{s}$}})_i \text{ belongs to the $j$-th terminal } \Gamma_{\scalebox{0.5}{$\mathrm{s}$}}^{(j)}\\
0,\quad \text{ otherwise.}
\end{cases}
\end{align*}
Here, $\Gamma_{\scalebox{0.5}{$\mathrm{s}$}} = \Gamma_{\scalebox{0.5}{$\mathrm{s}$}}^{(1)}\cup \ldots\cup\Gamma_{\scalebox{0.5}{$\mathrm{s}$}}^{(k)}$, for a $k$-port device, where
$$\Gamma_{\scalebox{0.5}{$\mathrm{s}$}}^{i}\cap\Gamma_{\scalebox{0.5}{$\mathrm{s}$}}^{j}=\emptyset,\; \text{ for }i\neq j\;.$$
With the junction \( Y_{\scalebox{0.5}{$\mathrm{s}$}} = P_{\scalebox{0.5}{$\mathrm{s}$}}\Lambda_{\scalebox{0.5}{$\mathrm{s}$}} \) the discrete gradient in \eqref{eq:descr_def} reads:
\begin{align*}
	\gradfit \bar{\Phi} = \gradfit \Qs \Phi + \gradfit Y_{\scalebox{0.5}{$\mathrm{s}$}} \vind\;.
\end{align*}
\begin{remark}\label{rem:Ysfullcol}
	Note that, as the different terminals $\Gamma_{\scalebox{0.5}{$\mathrm{s}$}}^{(j)}$ are disjoint,
	 per construction, $\Lambda_{\scalebox{0.5}{$\mathrm{s}$}}$, and therefore also $Y_{\scalebox{0.5}{$\mathrm{s}$}}$,
	 have full column rank.
\end{remark}
The spatially discretized version of \eqref{eq:pot_form} with incorporated boundary conditions \eqref{eq:pot_bc} is then given by
\begin{align}
	\Qs^\top\divdfit \Meps \gradfit \Qs \dxdy{}{t} {\Phi } + \Qs^\top\divdfit M_\zeta \gradfit M_\xi \divdfit  M_\zeta a  
	&= 0\;,\label{eq:MGE_potential_1}\\
	\curldfit  \Mnu C a  + 
	\dxdy{}{t} \left[
	\Meps \left(
	\gradfit \Qs \Phi + \gradfit Y_{\scalebox{0.5}{$\mathrm{s}$}} \vind +  {\pi} 
	\right)
	\right]
	+
	\Msigma\left(
	\gradfit \Qs \Phi + \gradfit Y_{\scalebox{0.5}{$\mathrm{s}$}} \vind + 	\dxdy{}{t} a  
	\right)
	&= 0\;,\label{eq:MGE_potential_2}\\
	\dxdy{}{t} a -{\pi} &= 0\;,\label{eq:MGE_potential_3}
\end{align}
where \( \pi \) is a discrete quasi-canonical momentum introduced in order to avoid second order derivatives. 
The discretized current coupling model of \eqref{eq:pot_cce} reads
\begin{align}\label{eq:boundary_excitation_pot}
\iind =& \Ys^{\top}\divdfit \curldfit  \Mnu \curlfit a \;.
\end{align}
For \( \xind = (\Phi, a, \pi) \), we define the system matrices
\begin{align*}
M:=& \begin{bmatrix}
\Qs^\top\divdfit \Meps\gradfit\Qs & 0&0\\
\Meps\gradfit\Qs & \Msigma & \Meps \\
0&I&0
\end{bmatrix},&
A :=& \begin{bmatrix}
0 & \Qs^\top\divdfit M_\zeta\gradfit M_\xi \divdfit  M_\zeta&0\\
\Msigma\gradfit\Qs& \curldfit \Mnu\curlfit & 0 \\
0&0&-I
\end{bmatrix},\\
N :=& \begin{bmatrix}0\\\Meps \gradfit\Ys \\ 0\end{bmatrix},&
B :=& \begin{bmatrix}0\\\Msigma \gradfit\Ys \\ 0\end{bmatrix},\\
F :=& \begin{bmatrix}0 & \Ys^{\top}\divdfit \curldfit  \Mnu \curlfit & 0\end{bmatrix}
\end{align*}
from which we conclude the EM device's element description
\begin{align}\label{eq:em_device}
\find\left(\dxdy{}{t}\dfind(\xind, \iind, \vind, t) , \xind, \iind, \vind,t\right) := \begin{pmatrix}
M \dxdy{}{t} \xind +A\xind +B \vind +N \dxdy{}{t}\vind\\\iind -F\xind
\end{pmatrix}=0\;.
\end{align}

\begin{proposition}
		Provided Assumptions~\ref{ass:materials}, \ref{ass:boundary} and \ref{ass:bcinmatrices} are fulfilled and the absence 
		of inner sources and Neumann boundary conditions, the EM device, whose model is given by the element description 
		\eqref{eq:em_device}, is a strongly inductance-like element.
\end{proposition}

\begin{proof}
The discrete gradient operator \( \gradfit \) and basis matrix \( \Qs \) have full column rank by 
Assumption \ref{ass:bcinmatrices} and construction \eqref{eq:phi_split}. 
Further it is \( \gradfit =  -\divdfit^\top \) and together with \( \Meps \) being positive definite, as of Assumption~\ref{ass:materials},
 we deduce that the Laplace-operator \( L_{\scalebox{0.5}{$\mathrm{Q}$}} := \Qs^\top \divdfit  \Meps \gradfit \Qs \) is non-singular. Hence, we find
\begin{align*}
	M^{-1} = \begin{bmatrix}
		L_{\scalebox{0.5}{$\mathrm{Q}$}}^{-1} & 0 & 0\\
		0 & 0 & I\\
		-\gradfit \Qs L_{\scalebox{0.5}{$\mathrm{Q}$}}^{-1} & \Meps^{-1} & -\Meps^{-1} \Msigma
	\end{bmatrix}.
\end{align*}
Therefore, we can define the following matrices 
\begin{align*}
	\widetilde A &:=M^{-1} A = \begin{bmatrix}
		0 & L_{\scalebox{0.5}{$\mathrm{Q}$}}^{-1}H & 0\\
		0 & 0 & -I\\
		\Meps^{-1} \Msigma\gradfit\Qs & -\gradfit \Qs L^{-1}H + \Meps^{-1} \curldfit \Mnu\curlfit & \Meps^{-1} \Msigma
	\end{bmatrix},\\
	\widetilde B &:=M^{-1}B = \begin{bmatrix}
		0\\
		0\\
		\Meps^{-1} \Msigma \gradfit \Ys
	\end{bmatrix}, \quad
	\widetilde N :=M^{-1}N = \begin{bmatrix}
		0\\
		0\\
		\gradfit \Ys
	\end{bmatrix}
\end{align*}
with \(H := \Qs^\top\divdfit M_\zeta\gradfit M_\xi\divdfit  M_\zeta \) and deduce from \eqref{eq:em_device} a description for \( \dxdy{}{t} \xind \) of the form \eqref{eq:ind.xf}
\begin{align}\label{eq:em-device.xind}
	\dxdy{}{t} \xind = -\widetilde A \xind - \widetilde B \vind - \widetilde N \dxdy{}{t} \vind =: \xfind(\dxdy{}{t}\vind,\xind,\iind,\vind,t)\;.
\end{align}
Next, we differentiate \eqref{eq:em_device} once, in particular the second part, and insert the expression for \( \dxdy{}{t}\xind \) from \eqref{eq:em-device.xind} yielding
\begin{align*} 
	\dxdy{}{t}\iind = F (-\widetilde A \xind - \widetilde B \vind - \widetilde N \dxdy{}{t} \vind) = -F\widetilde A \xind - F\widetilde B \vind =: \gind(\xind,\iind,\vind,t).
\end{align*}
Thus, we found an expression of \( \dxdy{}{t}\iind \) fitting \eqref{eq:ind.g}.
Finally, we observe that
\begin{align*}
	\strongfind(\vpind,\xind,\iind,\vind,t):=&\
	\partial_\xind \gind(\xind,\iind,\vind,t)\xfind(\vpind,\xind,\iind,\vind,t)\nonumber\\
	& +\partial_\vind \gind(\xind,\iind,\vind,t) \vpind\\
	=&\
	F \widetilde{A}  \widetilde A \xind + F \widetilde{A} \widetilde B \vind + F \widetilde{A} \widetilde N  \vpind - \underbrace{F\widetilde B}_{=0} \vpind
\end{align*}
is continuous and strongly monotone with respect to $\vpind$, see Lemma\  \ref{lem:posdef.monotone} using that \( F \widetilde{A} \widetilde{N} = -\Ys^{\top}\divdfit \curldfit \Mnu \curlfit \gradfit \Ys = \Ys^{\top}\gradfit^\top \curlfit^\top  \Mnu \curlfit \gradfit \Ys \) is positive definite by construction.
We conclude that this model for an EM device fulfills the strongly inductance-like property. \qed
\end{proof}

\begin{remark}\label{rmk:FMinvN}
	The fact that \( FM^{-1}N \) vanishes, as obtained by elemental matrix operations, plays a key role in the EM device's model fitting the inductance-like element description.
\end{remark}

For two different field approximations of Maxwell's equations that result in strongly inductance-like elements, see \cite{Cortes-Garcia_2019aa}. In contrast to our example, there the
strongly inductance-like element is given by the term $\partial_\vind \gind(\xind,\iind,\vind,t) \vpind$ in \eqref{eq:ind.strongf}, like in the case 
of classical  and flux-formulated inductances.

\subsection{Capacitance-like element}
We consider the electroquasistatic field approximation of Maxwell's equations \cite{Clemens_2006aa,Cortes-Garcia_2018aa}. As in this approximation, 
the electric field $\Efield$ is rotation free, we can write it in terms of only the
electric scalar potential $\phi$ \cite{Cortes-Garcia_2018aa}. 

Given a time-dependent excitation $\vcap$, we can write the following boundary value problem to describe
an electroquasistatic field
\begin{subequations}\label{eq:eqs}
\begin{align}
	\nabla\cdot\sigma\nabla\phi + \dxdy{}{t}\nabla\cdot\varepsilon\nabla\phi &= 0 && \text{ in }\Omega\;,\label{eq:eqs1}\\
																		\phi &=0 &&  \text{ in }\Gamma_{\scalebox{0.5}{$\mathrm{dir}$,0}}\;,\\
									{\partial_{\vec{n}}}\phi	&=0  &&  \text{ in }\Gamma_{\scalebox{0.5}{$\mathrm{neu}$,0}}\;,\\
																		\phi &=\vcap &&  \text{ in }\Gamma_{\scalebox{0.5}{$\mathrm{s}$}}\;,
\end{align}
\end{subequations}
with $\vec{n}$ being the outer normal vector to $\Gamma_{\scalebox{0.5}{$\mathrm{neu}$,0}}$.
To couple the electroquasistatic system \eqref{eq:eqs} to a circuit, the extraction of a current is necessary, so as to obtain
an implicit voltage-to-current relation. For that we integrate the current density \eqref{eq:eqs1} over the boundary, where the connections
to the circuit are located ($\Gamma_{\scalebox{0.5}{$\mathrm{s}$}}$), i.e.
\begin{equation}\label{eq:eqscoupling}
	\int_{\Gamma_{\scalebox{0.5}{$\mathrm{s}$}}} \left(\nabla\cdot\sigma\nabla\phi + \dxdy{}{t}\nabla\cdot\varepsilon\nabla\phi\right)\cdot\mathrm{d}\vec{S} = \icap\;.
\end{equation}
We assume first a spatial discretization of the PDEs \eqref{eq:eqs1} and \eqref{eq:eqscoupling} has been applied,
with only the boundary conditions
\begin{align}\label{eq:bcdirneucap}
	\phi =0 \quad  \text{ in }\Gamma_{\scalebox{0.5}{$\mathrm{dir}$,0}} && \text{ and } && {\partial_{\vec{n}}}\phi	=0  \quad  \text{ in }\Gamma_{\scalebox{0.5}{$\mathrm{neu}$,0}}\;.
\end{align} 
Analogously to the previous examples and given the homogeneous boundary conditions of \eqref{eq:bcdirneucap} are incorporated in the operator matrices, i.e., Assumption~\ref{ass:bcinmatrices} holds,
the spatially discretized electroquasitatic
field equation with circuit coupling equation is obtained as  \cite{Cortes-Garcia_2018aa}
\begin{subequations}\label{eq:discreteeqs}
\begin{align}
	{Q}_{\scalebox{0.5}{$\mathrm{s}$}}^{\top}L_{\sigma}{Q}_{\scalebox{0.5}{$\mathrm{s}$}}\Phi + 
			{Q}_{\scalebox{0.5}{$\mathrm{s}$}}^{\top}L_{\varepsilon}{Q}_{\scalebox{0.5}{$\mathrm{s}$}}\dxdy{}{t}\Phi + 
			{Q}_{\scalebox{0.5}{$\mathrm{s}$}}^{\top}L_{\sigma}Y_{\scalebox{0.5}{$\mathrm{s}$}}\vcap 
			+ {Q}_{\scalebox{0.5}{$\mathrm{s}$}}^{\top}L_{\varepsilon}Y_{\scalebox{0.5}{$\mathrm{s}$}}\dxdy{}{t}\vcap &=0\label{eq:discreteeqs.phi}\;,\\
	Y_{\scalebox{0.5}{$\mathrm{s}$}}^{\top}L_{\sigma}{Q}_{\scalebox{0.5}{$\mathrm{s}$}}\Phi +Y_{\scalebox{0.5}{$\mathrm{s}$}}^{\top}L_{\varepsilon}{Q}_{\scalebox{0.5}{$\mathrm{s}$}}\dxdy{}{t}\Phi + 
			Y_{\scalebox{0.5}{$\mathrm{s}$}}^{\top}L_{\sigma}Y_{\scalebox{0.5}{$\mathrm{s}$}}\vcap + 
			Y_{\scalebox{0.5}{$\mathrm{s}$}}^{\top}L_{\varepsilon}Y_{\scalebox{0.5}{$\mathrm{s}$}}\dxdy{}{t}\vcap &=\icap\label{eq:discreteeqs.i}\;,		
\end{align}
\end{subequations}
where $L_{\sigma} = \divdfit\Msigma\divdfit^{\top}$ and $L_{\varepsilon} =\divdfit\Meps\divdfit^{\top}$ are two Laplace matrices.

\begin{proposition}\label{prop:QsandYs}
	For ${Q}_{\scalebox{0.5}{$\mathrm{s}$}}$  and $Y_{\scalebox{0.5}{$\mathrm{s}$}}$, we have that
	\begin{equation*}
		{Q}_{\scalebox{0.5}{$\mathrm{s}$}}x_1\neq Y_{\scalebox{0.5}{$\mathrm{s}$}}x_2, \text{ for } x_1,x_2\neq 0\;.
	\end{equation*}
\end{proposition}
\begin{proof}
	This property follows directly from the definition of both matrices. We have $Y_{\scalebox{0.5}{$\mathrm{s}$}}x_2 = {P}_{\scalebox{0.5}{$\mathrm{s}$}}y_2$
	and, by construction, the image of ${P}_{\scalebox{0.5}{$\mathrm{s}$}}$ are the discrete elements living in $\Gamma_{\scalebox{0.5}{$\mathrm{s}$}}$,
	while the image of ${Q}_{\scalebox{0.5}{$\mathrm{s}$}}$ are the rest. Also, by construction, both matrices have full column rank and thus 
	a trivial kernel. \qed
\end{proof}

\begin{proposition}
	Provided Assumptions~\ref{ass:materials}, \ref{ass:boundary} and \ref{ass:bcinmatrices} are fulfilled, then the semidiscrete eletroquasistatic system of equations with 
	circuit coupling equation \eqref{eq:discreteeqs} is a strongly capacitance-like element.
\end{proposition}
\begin{proof}
	Due to  Assumptions~\ref{ass:bcinmatrices} and \ref{ass:materials}, and the fact that ${Q}_{\scalebox{0.5}{$\mathrm{s}$}}$ has full column rank, we start by rewriting \eqref{eq:discreteeqs.phi} as
	\begin{align}
		\dxdy{}{t}\Phi ={}& - (Q_{\scalebox{0.5}{$\mathrm{s}$}}^{\top}
									L_{\varepsilon}{Q}_{\scalebox{0.5}{$\mathrm{s}$}})^{-1}
								Q_{\scalebox{0.5}{$\mathrm{s}$}}^{\top}L_{\sigma} {Q}_{\scalebox{0.5}{$\mathrm{s}$}}\Phi\nonumber\\
						& -  ({Q}_{\scalebox{0.5}{$\mathrm{s}$}}^{\top}L_{\varepsilon}{Q}_{\scalebox{0.5}{$\mathrm{s}$}})^{-1}
								\left({Q}_{\scalebox{0.5}{$\mathrm{s}$}}^{\top}L_{\varepsilon}{Y}_{\scalebox{0.5}{$\mathrm{s}$}}\dxdy{}{t}\vcap
						 +{Q}_{\scalebox{0.5}{$\mathrm{s}$}}^{\top}L_{\sigma}{Y}_{\scalebox{0.5}{$\mathrm{s}$}}\vcap\right)\;.\label{eq:eqsdtphi}
	\end{align}
	Inserting this into \eqref{eq:discreteeqs.i} yields
	\begin{align}
		\icap = {}& Y_{\scalebox{0.5}{$\mathrm{s}$}}^{\top}
					\left(I -L_{\varepsilon}{Q}_{\scalebox{0.5}{$\mathrm{s}$}}
					({Q}_{\scalebox{0.5}{$\mathrm{s}$}}^{\top}L_{\varepsilon}{Q}_{\scalebox{0.5}{$\mathrm{s}$}})^{-1}
					Q_{\scalebox{0.5}{$\mathrm{s}$}}^{\top}\right)L_{\sigma}{Q}_{\scalebox{0.5}{$\mathrm{s}$}}\Phi\nonumber\\
		 & + Y_{\scalebox{0.5}{$\mathrm{s}$}}^{\top}\left(I- 
		 	L_{\varepsilon}{Q}_{\scalebox{0.5}{$\mathrm{s}$}}({Q}_{\scalebox{0.5}{$\mathrm{s}$}}^{\top}L_{\varepsilon}{Q}_{\scalebox{0.5}{$\mathrm{s}$}})^{-1}
		 				Q_{\scalebox{0.5}{$\mathrm{s}$}}^{\top}\right)L_{\sigma}{Y}_{\scalebox{0.5}{$\mathrm{s}$}}\vcap\nonumber\\
		 & + Y_{\scalebox{0.5}{$\mathrm{s}$}}^{\top}\left(L_{\varepsilon}-L_{\varepsilon}{Q}_{\scalebox{0.5}{$\mathrm{s}$}}
		 		({Q}_{\scalebox{0.5}{$\mathrm{s}$}}^{\top}L_{\varepsilon}{Q}_{\scalebox{0.5}{$\mathrm{s}$}})^{-1}
		 				Q_{\scalebox{0.5}{$\mathrm{s}$}}^{\top}L_{\varepsilon}\right){Y}_{\scalebox{0.5}{$\mathrm{s}$}}\dxdy{}{t}\vcap\;.\label{eq:eqsicapf}
	\end{align}
	Now we want to see that $C = Y_{\scalebox{0.5}{$\mathrm{s}$}}^{\top}\left(L_{\varepsilon}- 
	L_{\varepsilon}{Q}_{\scalebox{0.5}{$\mathrm{s}$}}({Q}_{\scalebox{0.5}{$\mathrm{s}$}}^{\top}L_{\varepsilon}{Q}_{\scalebox{0.5}{$\mathrm{s}$}})^{-1}
	Q_{\scalebox{0.5}{$\mathrm{s}$}}^{\top}L_{\varepsilon}\right){Y}_{\scalebox{0.5}{$\mathrm{s}$}}$ is positive definite.
	For that, using again that $L_{\varepsilon}$ is symmetric positive definite (Assumptions~\ref{ass:materials}~and~\ref{ass:bcinmatrices}) and thus its square root 
	exists and is also symmetric positive definite, we rewrite
	\begin{equation*}
		C = Y_{\mathrm{s}}^{\top}L_{\varepsilon}^{\frac{1}{2}}\left(I- 
		L_{\varepsilon}^{\frac{1}{2}}{Q}_{\scalebox{0.5}{$\mathrm{s}$}}({Q}_{\scalebox{0.5}{$\mathrm{s}$}}^{\top}L_{\varepsilon}{Q}_{\scalebox{0.5}{$\mathrm{s}$}})^{-1}
		Q_{\scalebox{0.5}{$\mathrm{s}$}}^{\top}L_{\varepsilon}^{\frac{1}{2}}\right)
		L_{\varepsilon}^{\frac{1}{2}}{Y}_{\scalebox{0.5}{$\mathrm{s}$}}.
	\end{equation*}
	It can easily be seen that $\left(I- 
	L_{\varepsilon}^{\frac{1}{2}}{Q}_{\scalebox{0.5}{$\mathrm{s}$}}({Q}_{\scalebox{0.5}{$\mathrm{s}$}}^{\top}
	L_{\varepsilon}{Q}_{\scalebox{0.5}{$\mathrm{s}$}})^{-1}Q_{\scalebox{0.5}{$\mathrm{s}$}}^{\top}L_{\varepsilon}^{\frac{1}{2}}\right)$
	is a symmetric projector and thus positive semidefinite. Therefore we have that $C$ is positive semidefinite. Let's assume that there exists a vector $x$ such that 
	$x^{\top}Cx=0$, then, 
	$$\left(I- 
	L_{\varepsilon}^{\frac{1}{2}}{Q}_{\scalebox{0.5}{$\mathrm{s}$}}({Q}_{\scalebox{0.5}{$\mathrm{s}$}}^{\top}L_{\varepsilon}{Q}_{\scalebox{0.5}{$\mathrm{s}$}})^{-1}
	Q_{\scalebox{0.5}{$\mathrm{s}$}}^{\top}L_{\varepsilon}^{\frac{1}{2}}\right)	L_{\varepsilon}^{\frac{1}{2}}{Y}_{\scalebox{0.5}{$\mathrm{s}$}}x = 0.$$
	However, this implies that 
	$$L_{\varepsilon}^{\frac{1}{2}}{Y}_{\scalebox{0.5}{$\mathrm{s}$}}x = 	
	L_{\varepsilon}^{\frac{1}{2}}{Q}_{\scalebox{0.5}{$\mathrm{s}$}}({Q}_{\scalebox{0.5}{$\mathrm{s}$}}^{\top}L_{\varepsilon}{Q}_{\scalebox{0.5}{$\mathrm{s}$}})^{-1}
	Q_{\scalebox{0.5}{$\mathrm{s}$}}^{\top}L_{\varepsilon}{Y}_{\scalebox{0.5}{$\mathrm{s}$}}x$$
	and multiplying this by $L_{\varepsilon}^{-\frac{1}{2}}$ would yield ${Y}_{\scalebox{0.5}{$\mathrm{s}$}}x = Q_{\scalebox{0.5}{$\mathrm{s}$}}y$, 
	with $y = ({Q}_{\scalebox{0.5}{$\mathrm{s}$}}^{\top}L_{\varepsilon}{Q}_{\scalebox{0.5}{$\mathrm{s}$}})^{-1}Q_{\mathrm{s}}^{\top}L_{\varepsilon}{Y}_{\scalebox{0.5}{$\mathrm{s}$}}x$. 
	Due to 
	Proposition~\ref{prop:QsandYs} this, however, is only possible if ${Y}_{\scalebox{0.5}{$\mathrm{s}$}}x=0$ and, as ${Y}_{\scalebox{0.5}{$\mathrm{s}$}}$ 
	has full column rank (see Remark~\ref{rem:Ysfullcol}),
	$x=0$. Therefore
	$C$ has full rank and is positive definite.
	
	According to Definition~\ref{def:cap-like}, we need to show that $\dxdy{}{t}\Phi$ can be written, with at most one differentiation,
	as a function depending only on $\dxdy{}{t}\icap,$ $\Phi$, $\vcap$, $\icap$ and $t$ (see~\eqref{eq:cap.xf}). 
	For that we invert $C$ in \eqref{eq:eqsicapf} to obtain 
	\begin{equation}\label{eq:eqsdtv}
		\dxdy{}{t}\vcap = \gcap(\Phi, \icap,\vcap)\;.
	\end{equation}
	This can now be inserted into \eqref{eq:eqsdtphi} to obtain a function
	\begin{equation}
	\dxdy{}{t}\Phi = \xfcap (\Phi, \icap, \vcap)\;,
	\end{equation}
	without having required any differentiation of the original system. 
	
	Due to \eqref{eq:eqsdtv}, we have already shown that we obtain a capacitance-like element. Furthermore, as $\partial_\icap \gcap(\Phi, \icap,\vcap)=C$, is
	positive definite, using Lemma~\ref{lem:posdef.monotone} and Definition~\ref{def:stronmonomxn}, the system is shown to be strongly capacitance-like. \qed

\end{proof}

\subsection{Resistance-like element}
The last refined model we study is the  eddy current equation for the simulation of magnets with superconducting coils. For that we consider
a magnetoquasistatic approximation of Maxwell's equations \cite{Jackson_1998aa} in terms of the $\vec{A}^*$ formulation \cite{Emson_1988aa}. 
Here, the gauging freedom of the magnetoquasistatic
setting allows to choose a special magnetic vector potential $\vec{A}$, such that the electric scalar potential $\varphi$ vanishes from the PDE. The governing equation
reads
\begin{align*}
	\nabla\times\nu\tau_{\scalebox{0.5}{$\mathrm{eq}$}}\nabla\times\dxdy{}{t}\Afield - \nabla\times\nu\nabla\times \Afield = \Jfield_{\scalebox{0.5}{$\mathrm{s}$}}\;.  
\end{align*}
The non-standard expression $\nabla\times\nu\tau_{\scalebox{0.5}{$\mathrm{eq}$}}\nabla\times\dxdy{}{t}\Afield$ is an homogenization 
model accounting for the cable magnetization, that represents
the eddy current effects of the
superconducting coils \cite{De-Gersem_2004ae}. It contains the cable time constant $\tau_{\scalebox{0.5}{$\mathrm{eq}$}}$,
 which depends on certain properties of the cable \cite{Verweij_1995aa}. 
This formulation is coupled to a circuit in order to simulate the superconducting magnet's protection system of the
LHC at CERN \cite{Bortot_2018ab,Cortes-Garcia_2017ab}. For the boundary value problem we also set the boundary conditions
\begin{equation}
	\vec{n}\times\Afield = 0, \quad \text{on }\Gamma_{\scalebox{0.5}{$\mathrm{dir}$,0}} \qquad \text{and} \qquad
	 \vec{n}\times(\nu\nabla\times\Afield) = 0, \quad \text{on }\Gamma_{\scalebox{0.5}{$\mathrm{neu}$,0}}\;,
\end{equation}
 where $\vec{n}$ is again the outer normal vector to the boundary $\Gamma$. Please note that here, no
 boundary conditions where set on $\Gamma_{\scalebox{0.5}{$\mathrm{s}$}}$, as for this example
 $\Gamma_{\scalebox{0.5}{$\mathrm{s}$}} = \emptyset$.

In this case, as $\Gamma_{\scalebox{0.5}{$\mathrm{s}$}} = \emptyset$, the circuit coupling is not performed through the boundary but by 
a characteristic function (winding density function) \cite{Schops_2013aa}, that discributes the zero dimensional current $\ires$ on the
two or three dimensional domain of the PDE. For the excitation of the coil's 
cross-section we define a $\boldsymbol{\chi}_{\scalebox{0.5}{$\mathrm{s}$}}:\Omega\rightarrow\setR^3$, such that
\begin{equation*}
	\vec{J}_{\scalebox{0.5}{$\mathrm{s}$}} = \boldsymbol{\chi}_{\scalebox{0.5}{$\mathrm{s}$}} \ires\;.
\end{equation*}
This also allows to extract the voltage across the coil as
\begin{equation*}
	\vres = -\int_{\Omega}\boldsymbol{\chi}_{\scalebox{0.5}{$\mathrm{s}$}}\cdot\vec{E}\;\mathrm{d}V\;.
\end{equation*}

\begin{assumption}\label{ass:tausource}
	As the magnet is excited through the superconducting coils, we assume that the domain, where the source current density is nonzero
	also corresponds to the domain, where the cable time constant is positive, that is
	$$\sup\tau_{\scalebox{0.5}{$\mathrm{eq}$}} = \sup\boldsymbol{\chi}_{\scalebox{0.5}{$\mathrm{s}$}} = \Omega_{\scalebox{0.5}{$\mathrm{s}$}}\;.$$
\end{assumption}

After spatial discretisation of the eddy current PDE with coupling equation, we obtain the DAE
\begin{subequations}\label{eq:discretemqs}
	\begin{align}
	\curlfit^{\top} M_{\nu,\tau_{\mathrm{eq}}}\curlfit\dxdy{}{t} a + \curlfit^{\top}\Mnu\curlfit a &= X\ires\label{eq:discretemqs1}\\
	X^{\top}\dxdy{}{t}a&=\vres\;,
	\end{align}
\end{subequations}
where $X$ is a vector, containing the discretisation of the winding density function. We define the orthogonal projector $Q_{\tau}$ onto $\ker \curlfit^{\top} M_{\nu,\tau_{\mathrm{eq}}}\curlfit$
and its complementary $P_{\tau} = I-Q_{\tau}$.

\begin{assumption}\label{ass:Cdiscrete}
	We assume that 
	\begin{itemize}
		\item the curl matrix $\curlfit$ and the discrete magnetic vector potential $a$ are gauged and contain homogeneous Dirichlet boundary conditions,
		such that $\curlfit$ has full column rank.
		\item there is no excitation outside of the coils, i.e., $Q_{\tau}^{\top}X=0$.
	\end{itemize}
\end{assumption}
The first part of the assumption is necessary, such that the DAE system \eqref{eq:discretemqs} is uniquely solvable. This is possible by for example using a 
tree-cotree gauge \cite{Munteanu_2002aa}, where the degrees of freedom of $a$ belonging to a gradient field are eliminated. The second part of the assumption
is motivated by the fact that the source current density has to be divergence-free, together with Assumption~\ref{ass:tausource}.

\begin{proposition}
	Provided Assumptions~\ref{ass:materials},~\ref{ass:tausource}-\ref{ass:Cdiscrete} are fulfilled, then the semidiscrete homogenized eddy current system of equations with 
	circuit coupling equation \eqref{eq:discretemqs} is a strongly resistance-like element.
\end{proposition}
\begin{proof}
	We start by multiplying equation \eqref{eq:discretemqs1} by $Q_{\tau}^{\top}$ and $P_{\tau}^{\top}$ and obtain
	\begin{subequations}
		\begin{align}
		\curlfit^{\top} M_{\nu,\tau_{\mathrm{eq}}}\curlfit\dxdy{}{t} a + P_{\tau}^{\top}\curlfit^{\top}\Mnu\curlfit a &= P_{\tau}^{\top}X\ires\label{eq:Pmqs}\\
		Q_{\tau}^{\top}\curlfit^{\top}\Mnu\curlfit a &= Q_{\tau}^{\top}X\ires\label{eq:Qmqs}
		\end{align}
	\end{subequations}
	From \eqref{eq:Pmqs} we obtain without the need of any differentiation
	\begin{equation}\label{eq:Pdta}
		P_{\tau}\dxdy{}{t}a = (\curlfit^{\top} M_{\nu,\tau_{\mathrm{eq}}}\curlfit + Q_{\tau}^{\top}Q_{\tau})^{-1}(P_{\tau}^{\top}X\ires - P_{\tau}^{\top}\curlfit^{\top}\Mnu\curlfit a)\;.
	\end{equation}
	Differentiating \eqref{eq:Qmqs} once and using Assumption~\ref{ass:Cdiscrete} we have
	 \begin{align}\label{eq:Qdta}
	 	Q_{\tau}\dxdy{}{t}a = -(Q_{\tau}^{\top} \curlfit^{\top}\Mnu\curlfit Q_{\tau} + P_{\tau}^{\top}P_{\tau})^{-1}Q_{\tau}^{\top}\curlfit^{\top}\Mnu\curlfit  P_{\tau}\dxdy{}{t}a
	 \end{align}
	 Inserting \eqref{eq:Pdta} into \eqref{eq:Qdta} we obtain an ODE with the structure
	 \begin{equation*}
	 	\dxdy{}{t}\xres = \xfres(\xres, \ires)\;,
	 \end{equation*}
	 where $\xres = P_{\tau}a + Q_{\tau}a$.
	 Now we use Assumption~\ref{ass:Cdiscrete} and insert \eqref{eq:Pdta} into the circuit coupling equation to obtain
	 \begin{align*}
	 	\vres & =X^{\top}\dxdy{}{t}(P_{\tau}a + Q_{\tau}a) = X^{\top}\dxdy{}{t}P_{\tau}a\\
	 	     &= X^{\top}P_{\tau}(\curlfit^{\top} M_{\nu,\tau_{\mathrm{eq}}}\curlfit + Q_{\tau}^{\top}Q_{\tau})^{-1}(P_{\tau}^{\top}X\ires - P_{\tau}^{\top}\curlfit^{\top}\Mnu\curlfit a)\;.
	 \end{align*}
	 To obtain this expression no differentiation was needed, thus if we differentiate it once, 
	 according to  Definition~\ref{def:res-like} and using  Lemma~\ref{lem:posdef.monotone} and Definition~\ref{def:stronmonomxn}, 
	 we now only need to show that $G = \partial_{\vres'}\gres(\vres',\xres,\ires,\vres,t)$, with 
	 $$G^{-1} = X^{\top}P_{\tau}(\curlfit^{\top} M_{\nu,\tau_{\mathrm{eq}}}\curlfit + Q_{\tau}^{\top}Q_{\tau})^{-1}P_{\tau}^{\top}X\;,$$ 
	 is positive definite to obtain that \eqref{eq:discretemqs}
	 is a strongly resistance-like element. This follows immediately by the fact that $M_{\nu,\tau_{\mathrm{eq}}}$ 
	 is positive semidefinite (Assumptions~\ref{ass:materials}~and~\ref{ass:tausource}) and $X$ has full column rank, as it is only a vector.\qed
\end{proof}

\section{Conclusions}
This paper has demonstrated that even very complicated refined models with internal degrees of freedom can be characterized by generalizations of the basic circuit elements, i.e., resistance, inductance and capacitance. This knowledge significantly simplifies the structural analysis of future networks consisting of refined models. Structural properties of the network, e.g. the differential algebraic index, can easily be deduced if the element is identified in terms of the proposed generalized elements.

\printindex

\bibliographystyle{spmpsci}

\end{document}